\documentclass{elsarticle}

\usepackage{amsmath,amsfonts, amssymb,amsthm}
\usepackage{hyperref}
\hypersetup{colorlinks=false,pdftitle="Log-concavity and compound Poissons",pdftex}

\begin{document}

\begin{frontmatter}
\title{Log-concavity, ultra-log-concavity, and a maximum entropy property\\
of discrete compound Poisson measures}

\author[bristol]{Oliver Johnson} 
\ead{O.Johnson@bristol.ac.uk} 
\author[athens]{Ioannis Kontoyiannis\fnref{fn2}} 
\ead{yiannis@aueb.gr} 
\author[yale]{Mokshay Madiman\corref{cor1}\fnref{fn3}} 
\ead{mokshay.madiman@yale.edu}

\cortext[cor1]{Corresponding author.} 
\fntext[fn2]{I. Kontoyiannis was supported in part
   by a Marie Curie International
   Outgoing Fellowship, PIOF-GA-2009-235837.} 
\fntext[fn3]{M. Madiman  was supported by a Junior Faculty Fellowship from Yale University in spring 2009,
as well as by the CAREER grant DMS-1056996 and by CCF-1065494 from the U.S. National Science Foundation.
} 

\address[bristol]{Department of Mathematics, University of Bristol, University Walk, Bristol, BS8 1TW, UK.} 
\address[athens]{Department of Informatics,
Athens University of Economics \& Business,
Patission 76, Athens 10434, Greece.} 
\address[yale]{Department of Statistics, Yale University,
24 Hillhouse Avenue, New Haven, CT 06511, USA.}

\begin{abstract} 
Sufficient conditions are developed,
under which the compound Poisson distribution 
has maximal entropy within a natural class 
of probability measures on the nonnegative integers. 
Recently, one of the authors [O. Johnson, {\em Stoch. Proc. Appl.}, 2007] %
used a semigroup approach 
to show that the Poisson has maximal entropy among
all ultra-log-concave distributions with fixed mean.
We show via a non-trivial extension of this semigroup
approach that the natural analog of the Poisson 
maximum entropy property remains valid if the
compound Poisson distributions under consideration 
are log-concave,
but that it fails in general.
A parallel maximum entropy result is established
for the family of compound binomial measures.
Sufficient conditions for compound 
distributions to be log-concave are discussed
and applications to combinatorics are examined;
new bounds are derived on the entropy of the 
cardinality of a random independent set 
in a claw-free graph, and a connection is drawn
to Mason's conjecture for matroids.
The present results are primarily
motivated by the desire to provide an
information-theoretic foundation for 
compound Poisson approximation and associated 
limit theorems,
analogous to the corresponding developments
for the central limit theorem and for 
Poisson approximation.
Our results also demonstrate new links between 
some probabilistic methods and the combinatorial 
notions of log-concavity and ultra-log-concavity,
and they add to the growing body of work 
exploring the applications of maximum entropy
characterizations to problems in discrete mathematics. 
\end{abstract}


\end{frontmatter}

\newtheorem{theorem}{Theorem}[section]
\newtheorem{lemma}[theorem]{Lemma}
\newtheorem{proposition}[theorem]{Proposition}
\newtheorem{corollary}[theorem]{Corollary}
\newtheorem{conjecture}[theorem]{Conjecture}
\newtheorem{definition}[theorem]{Definition}
\newtheorem{example}[theorem]{Example}
\newtheorem{condition}{Condition}
\newtheorem{main}{Theorem}
\newtheorem{remark}[theorem]{Remark}
\hfuzz25pt 



\newcommand{\vc}[1]{{\mathbf{ #1}}}
\newcommand{\bp}[1]{b_{\vc{#1}}}
\newcommand{\cp}[1]{C_Q b_{\vc{#1}}}
\newcommand{\cbin}[1]{C_Q {\rm Bin}_{#1}}
\newcommand{\cPsa}{C_Q P_{\alpha}^{\#}}
\newcommand{\cP}{C_Q P}
\newcommand{\cPa}{C_Q P_{\alpha}}
\newcommand{\cPx}{C_Q P_{X}}
\newcommand{\cPy}{C_Q P_{Y}}
\newcommand{\cPz}{C_Q P_{X+Y}}
\newcommand{\sco}[1]{r_{1,#1}}
\newcommand{\Pc}{{\mathcal{ P}}}
\newcommand{\wt}[1]{\widetilde{#1}}
\newcommand{\var}{{\rm{Var\;}}}
\newcommand{\cov}{{\rm{Cov\;}}}
\newcommand{\tends}{\rightarrow \infty}
\newcommand{\C}{{\cal C}}
\newcommand{\ep}{{\mathbb {E}}}
\newcommand{\pr}{{\mathbb {P}}}
\newcommand{\re}{{\mathbb {R}}}
\newcommand{\I}{\mathbb {I}}
\newcommand{\Z}{{\mathbb {Z}}}
\newcommand{\tq}{\widetilde{Q}}
\newcommand{\tc}{\widetilde{C}}
\newcommand{\qs}[1]{Q^{\# (#1)}}
\newcommand{\qst}[1]{Q^{* #1}}
\newcommand{\qsh}{Q^{\#}}
\newcommand{\Nat}{{\mathbb {N}}}
\newcommand{\map}[1]{{\bf #1}}
\newcommand{\CPi}{C_Q \Pi}
\newcommand{\pbar}{\vc{\overline{p}}}

\newcommand{\leqa}{\mbox{$ \;\stackrel{(a)}{\leq}\; $}}

\newcommand{\Ch}[2]{\ensuremath{\begin{pmatrix} #1 \\ #2 \end{pmatrix}}} 
\newcommand{\bin}[2]{\binom{#1}{#2}}
\newcommand{\ov}[1]{\overline{#1}}
\newcommand{\bern}[1]{{\rm{Bern}}\left(#1\right)}
\newcommand{\cbern}[2]{{\rm{CBern}}\left(#1,#2\right)}
\newcommand{\geo}[1]{{\rm{Geom}}\left(#1\right)}
\newcommand{\Prob}[1]{\ensuremath{\mathbb{P} \left(#1 \right)}}
\newcommand{\blah}[1]{}
\newcommand{\ch}[1]{{\bf #1}}

\newcommand{\equald}{\mbox{$ \;\stackrel{\cal D}{=}\; $}}
\newcommand{\Zpl}{\mathbb{Z}_{+}}
\newcommand{\eqD}{\mbox{$ \;\stackrel{\mathcal{D}}{=}\; $}}

\section{Introduction}

The primary motivation for this work is the
development of an information-theoretic approach to discrete limit laws,
specifically those corresponding to compound Poisson limits.
Recall that the classical central limit theorem
can be phrased as follows: If $X_1,X_2,\ldots$ 
are independent and identically distributed,
continuous random variables with zero mean and unit variance, 
then the entropy of their normalized partial sums
$S_n=\frac{1}{\sqrt{n}}\sum_{i=1}^nX_i$ increases 
with $n$ to
the entropy of the standard normal distribution, 
which is maximal among all random variables with
zero mean and unit variance. More precisely, if
$f_n$ denotes the density of $S_n$ and $\phi$ 
the standard normal density, then, as $n\to\infty$,
\begin{equation}
h(f_n)\uparrow h(\phi)
=\sup\{h(f)\;:\;\mbox{densities $f$ with mean 0 and
variance 1}\},
\label{eq:clt}
\end{equation}
where $h(f)=-\int f\log f$ denotes the differential 
entropy, and `log' denotes the natural logarithm.
Precise conditions under which 
(\ref{eq:clt}) holds are given by Barron \cite{barron} and Artstein et al. \cite{artstein};
also see \cite{johnson14,tulino,madiman}
and the references therein.

Part of the appeal of this formalization of the
central limit theorem comes from its analogy
to the second law of thermodynamics: The
``state'' (meaning the distribution)
of the random variables $S_n$ evolves
monotonically, until the {\em maximum entropy}
state, the standard normal distribution, is
reached. Moreover, the introduction of
information-theoretic ideas and techniques
in connection with the entropy has motivated 
numerous related results (and their proofs),
generalizing and strengthening the central
limit theorem in different directions; see
the references above for details.

Recently, some discrete limit laws have been 
examined in a similar vein, but, as the discrete entropy 
$H(P)=-\sum_x P(x)\log P(x)$ for probability mass 
functions $P$ on a countable set naturally 
replaces the differential entropy $h(f)$, many of the relevant analytical 
tools become unavailable. 
For Poisson convergence theorems, 
of which the binomial-to-Poisson is the prototypical 
example, an analogous program has been carried out in
\cite{shepp,harremoes,johnson11,johnson21}.
Like with the central limit theorem, there are two aspects 
to this theory -- the Poisson distribution is first
identified as that which has maximal entropy within 
a natural class of probability measures,
and then convergence of appropriate sums of random variables
to the Poisson is established in the sense of relative entropy 
(or better still, approximation bounds are obtained 
that quantify the rate of this convergence).

One of the main goals of this work is to establish
a starting point for developing an 
information-theoretic framework for the
much more general class of {\em compound Poisson} 
limit theorems.\footnote{
Recall that the compound Poisson distributions 
are the only infinitely divisible distributions 
on ${\mathbb Z}_+$, and also
they are (discrete) 
stable laws \cite{steutel2}.
In the way of motivation,
we may also recall the remark of Gnedenko 
and Korolev \cite[pp. 211-215]{gnedenko2}
that ``there should be mathematical \ldots
probabilistic models of the universal principle
of non-decrease of uncertainty,''
and their proposal that we should
``find conditions under which certain limit 
laws appearing in limit theorems of probability 
theory possess extremal entropy properties. Immediate 
candidates to be subjected to such analysis are, 
of course, stable laws.''}
To that end, our first main result, 
given in Section~2, provides a maximum 
entropy characterization of compound Poisson laws,
generalizing Johnson's characterization 
\cite{johnson21} of the Poisson distribution.
It states that if one looks at the class
of all ultra-log-concave distributions 
on ${\mathbb Z}_+$ with a fixed mean,
and then compounds each distribution 
in this class using a given 
probability measure on ${\mathbb N}$, 
then the compound Poisson has maximal
entropy in this class, provided it is log-concave.

Having established conditions under which
a compound Poisson measure has maximum entropy,
in a companion work \cite{johnson22} we consider
the problem of establishing compound Poisson
limit theorems as well as finite-$n$ approximation
bounds, in relative entropy and total variation.
The tools developed in the present work,
and in particular the definition and analysis
of a new score function in Section~3,
play a crucial role in these compound Poisson
approximation results.


In a different direction, in Section~6 we demonstrate how the present results provide
new links between classical probabilistic methods and the combinatorial notions
of log-concavity and ultra-log-concavity. Log-concave sequences are well-studied 
objects in combinatorics, see, e.g., the surveys by Brenti \cite{Bre89:book} and Stanley \cite{Sta89}.
Additional motivation in recent years has come from the search
for a theory of negative dependence. Specifically, as argued by Pemantle \cite{pemantle}, 
a theory of negative dependence has long been desired in probability and statistical physics,
in analogy with the theory of positive dependence exemplified 
by the Fortuin-Kasteleyn-Ginibre (FKG) inequality \cite{FKG71}
(earlier versions were developed by Harris \cite{Har60} and,
in combinatorics, by Kleitman \cite{Kle66}).
But the development of such a theory is believed to be difficult and delicate.
For instance, Wagner \cite{wagner08} recently formalized a folklore conjecture
in probability theory (called by him the ``Big Conjecture'')
which asserts that, if a probability measure on the Boolean hypercube 
satisfies certain negative correlation conditions, then the sequence $\{p_k\}$ 
of probabilities of picking a set of size $k$, is ultra-log-concave. 
This is closely related to Mason's conjecture for independent sets in matroids, which
asserts that the sequence $\{I_k\}$, where $I_k$ is the number of independent 
sets of size $k$ in a matroid on a finite ground set, is ultra-log-concave.
Soon after, the ``Big Conjecture'' was falsified, by Borcea, Branden and Liggett \cite{BBL09}
and by Kahn and Neiman \cite{KN10}, who independently produced counterexamples.
In the other direction, very recently (while we were revising this paper), 
Lenz \cite{Len11} proved a weak version of Mason's conjecture.
In Section~6 we describe some simple consequences
of our results in the context of matroid theory,
and we also discuss an application to bounding the 
entropy of the size of a random independent
set in a claw-free graph.


Before stating our main results in detail, 
we briefly mention how this line of work connects with 
the growing body of work exploring applications of maximum entropy
characterizations to discrete mathematics. 
The simplest maximum entropy result states that, 
among all probability distributions on a finite 
set $S$, the uniform has maximal entropy, 
$\log |S|$. While mathematically trivial, 
this result, combined with appropriate structural 
assumptions and various entropy inequalities, 
has been employed as a powerful tool
and has seen varied applications in combinatorics. 
Examples include Radhakrishnan's 
entropy-based proof \cite{radha97} of Bregman's
theorem on the permanent of a 0-1 matrix,
Kahn's proof \cite{Kah01a} of the result 
of Kleitman and Markowski \cite{KM75}
on Dedekind's problem concerning the number of antichains 
in the Boolean lattice,
the study by Brightwell and Tetali \cite{BT03} of the number of 
linear extensions of the Boolean lattice (partially confirming a 
conjecture of Sha and Kleitman \cite{SK87}),
and the resolution of several conjectures of Imre Ruzsa 
in additive combinatorics by 
Madiman, Marcus and Tetali \cite{MMT08}.
However, so far, a limitation of this line of work has been 
that it can only handle
problems on finite sets.
As modern combinatorics explores
more and more properties of countable structures --
such as infinite graphs or posets --
it is natural that analogues of useful tools 
such as maximum entropy characterizations 
in countable settings should develop in parallel.
It is particularly natural to develop these in connection with 
the Poisson and compound Poisson laws, which
arise naturally in probabilistic combinatorics; see, e.g., 
Penrose's work \cite{Pen03:book} on geometric random graphs.

Section~2 contains
our two main results: The
maximum entropy characterization of log-concave 
compound Poisson distributions,
and an analogous result for compound binomials.
Sections~3 and~4, respectively, provide their proofs.
Section~5 discusses conditions for log-concavity,
and gives some additional results.
Section~6 discusses applications 
to classical combinatorics,
including graph theory and matroid theory. 
Section~7 contains some concluding remarks,
a brief description of potential generalization and
extensions, and a discussion of recent, subsequent
work by Y.\ Yu \cite{Yu09:cp}, which was motivated by
preliminary versions of some of the present results.

\section{Maximum Entropy Results}

First we review the maximum entropy property of the Poisson distribution.

\begin{definition}
\label{def:bp}
For any parameter vector $\vc{p} = (p_1,p_2, \ldots, p_n)$ 
with each $p_i\in[0,1]$,
the sum of independent Bernoulli random variables $B_i\sim\bern{p_i}$,
$$S_n=\sum_{i=1}^n B_i,$$
is called a {\em Bernoulli sum}, and its 
probability mass function is denoted by 
$\bp{p}(x):=\Pr\{S_n=x\}$,
for $x=0,1,\ldots$. Further, for each $\lambda>0$, we define
the following sets of parameter vectors: 
$$
\Pc_n(\lambda) \; = \; \big\{ \vc{p}\in[0,1]^n
\;:\; p_1+p_2+\cdots+p_n =\lambda 
\big\}
\;\;\;\;
\mbox{and}
\;\;\;\;
\Pc_{\infty}(\lambda) = \bigcup_{n\geq 1} \Pc_n(\lambda).
$$
\end{definition}
Shepp and Olkin \cite{shepp} showed 
that, for fixed $n\geq1$,
the Bernoulli sum $\bp{p}$ which has maximal 
entropy among all Bernoulli sums with
mean $\lambda$,
is Bin$(n,\lambda/n)$,
the binomial with parameters $n$ and $\lambda/n$,
\begin{equation}
H(\mbox{Bin}(n,\lambda/n))
=
\max\Big\{ H(\bp{p})\;:\; {\vc{p}\in\Pc_n(\lambda)}\Big\},
\label{eq:maxEntB}
\end{equation}
where $H(P)=-\sum_x P(x)\log P(x)$ denotes the discrete
entropy function. Noting that the binomial
$\mbox{Bin}(n,\lambda/n)$ converges to the Poisson
distribution $\mbox{Po}(\lambda)$ as $n\to\infty$,
and that the classes of Bernoulli sums in (\ref{eq:maxEntB})
are nested, 
$\{\bp{p}:\vc{p}\in\Pc_n(\lambda)\}\subset
\{\bp{p}:\vc{p}\in\Pc_{n+1}(\lambda)\},$ 
Harremo\"es \cite{harremoes} 
noticed that a simple limiting
argument gives the following 
maximum entropy property
for the Poisson distribution:
\begin{equation}
H(\mbox{Po}(\lambda))
=
\sup\Big\{
H(\bp{p})\;:\;
\vc{p}\in\Pc_\infty(\lambda)\Big\}.
\label{eq:maxEntP}
\end{equation}

A key property in generalizing and understanding
this maximum entropy property further 
is that of
ultra-log-concavity;
cf.\ \cite{pemantle}. The distribution $P$ of a random variable 
$X$ is {\em ultra-log-concave} if $P(x)/\Pi_{\lambda}(x)$ is 
log-concave, that is, if,
\begin{equation} \label{eq:ulcdef}
x P(x)^2 \geq (x+1) P(x+1) P(x-1),\;\;\;\; \mbox{for all $x \geq 1$.}
\end{equation}
Note that the Poisson distribution as well as all Bernoulli sums 
are ultra-log-concave. A non-trivial property of the class of 
ultra-log-concave distributions, conjectured by Pemantle \cite{pemantle}
and proved by Liggett \cite{Lig97} (cf. Gurvits \cite{Gur09} and Kahn and Neiman \cite{KN11}),
is that it is closed under convolution.

Johnson \cite{johnson21} recently
proved the following maximum entropy 
property for the Poisson distribution,
generalizing (\ref{eq:maxEntP}):
\begin{equation}
H(\mbox{Po}(\lambda))
=
\max\Big\{
H(P)\;:\;
\mbox{ultra-log-concave $P$ with mean $\lambda$}
\Big\}.
\label{eq:maxEntPJ}
\end{equation}


As discussed in the Introduction, we wish to generalize
the maximum entropy
properties (\ref{eq:maxEntB})
and (\ref{eq:maxEntP}) to the case of
{\em compound Poisson} distributions
on ${\mathbb Z}_+$.
We begin with some definitions:

\begin{definition} Let $P$ be an arbitrary distribution
on $\Z_+=\{0,1,\ldots\}$, and $Q$ a distribution on 
$\Nat = \{1, 2, \ldots \}$.
The {\em $Q$-compound distribution $\cP$} is the
distribution of the random sum,
\begin{equation} \label{eq:randsum}
\sum_{j=1}^{Y} X_j,
\end{equation}
where $Y$ has distribution $P$ and the random variables 
$\{X_j\}$ are independent and identically distributed
(i.i.d.) with common distribution $Q$ and 
independent of $Y$.
The distribution $Q$ is called a
{\em compounding distribution},
and the map $P\mapsto C_Q P$ is
the {\em $Q$-compounding operation}.
The $Q$-compound distribution $C_QP$
can be explicitly written as the mixture,
\begin{equation} 
\label{eq:compdis}
\cP(x) = \sum_{y=0}^{\infty} P(y) \qst{y}(x),
	\;\;\;\;x\geq 0,
\end{equation}
where $Q^{*j}(x)$ is the $j$th convolution power of $Q$ and 
$Q^{*0}$ is the point mass at $x=0$. 
\end{definition}

Above and throughout the paper,
the empty sum $\sum_{j=1}^0(\cdots)$ is taken to be zero;
all random variables considered are supported 
on $\Z_+=\{0,1,\ldots\}$; and all compounding 
distributions $Q$ are supported on $\Nat=\{1,2,\ldots\}$.

\begin{example} Let $Q$ be an arbitrary distribution on $\Nat$.
\begin{enumerate} 
	\item 
For any $0 \leq p \leq 1$, the {\em compound Bernoulli
distribution $\cbern{p}{Q}$} is the distribution
of the product $BX$, where $B\sim\mbox{Bern}(p)$
and $X\sim Q$ are independent.
It has probability mass function
$C_Q P$, where $P$ is the $\bern{p}$ mass function,
so that, $C_Q P(0)=1-p$ and $C_Q P(x)=pQ(x)$ for $x\geq 1$.
	\item 
A {\em compound Bernoulli sum} is a sum of independent 
compound Bernoulli random variables, all with respect 
to the same compounding distribution $Q$: Let 
$X_1,X_2,\ldots,X_n$ be i.i.d.\ with common
distribution $Q$ and $B_1,B_2,\ldots,B_n$
be independent Bern($p_i$). We call,
$$ \sum_{i=1}^n B_iX_i \;\equald\; \sum_{j=1}^{\sum_{i=1}^n B_i} X_j,$$
a {\em compound Bernoulli sum}; in view of~{\em (\ref{eq:randsum})},
its distribution is $\cp{p}$, where
$\vc{p} = (p_1,p_2, \ldots, p_n)$.
	\item
In the special case of a compound Bernoulli sum with
all its parameters $p_i=p$ for a fixed $p\in[0,1]$,
we say that it has a {\em compound binomial distribution},
denoted by $\mbox{\em CBin}(n,p,Q)$.
	\item
Let $\Pi_\lambda(x)=e^{-\lambda}\lambda^x/x!$, $x\geq 0$,
denote the {\em Po}$(\lambda)$ mass function. Then,
for any $\lambda>0$,
the {\em compound Poisson distribution $\mbox{CPo}(\lambda,Q)$}
is the distribution with mass function $\CPi_\lambda$:
\begin{equation} \label{eq:cppmf}
\CPi_{\lambda}(x) =
\sum_{j=0}^{\infty} \Pi_\lambda(j)
Q^{*j}(x) = 
\sum_{j=0}^{\infty} \frac{ e^{-\lambda} \lambda^j}{j!}
Q^{*j}(x),
\;\;\;\;x\geq 0.
\end{equation}
\end{enumerate}
\end{example}

In view of the Shepp-Olkin maximum entropy property (\ref{eq:maxEntB})
for the binomial distribution, a first natural conjecture
might be that the compound binomial has maximum entropy
among all compound Bernoulli sums $\cp{p}$
with a fixed mean; that is, 
\begin{equation}
H(\mbox{CBin}(n,\lambda/n,Q))
=
\max\Big\{ H(C_Q\bp{p})\;:\; {\vc{p}\in\Pc_n(\lambda)}\Big\}.
\label{eq:maxEntBC}
\end{equation}
But, perhaps somewhat surprisingly, as Zhiyi Chi \cite{chi} 
has noted, (\ref{eq:maxEntBC}) fails in general. For example,
taking $Q$ to be the uniform distribution on $\{1,2\}$,
$\vc{p}=(0.00125, 0.00875)$
and $\lambda =p_1+p_2=0.01$, 
direct computation shows that,
\begin{equation}
H(\mbox{CBin}(2,\lambda/2,Q))
<0.090798
<0.090804
< H(C_Q\bp{p}).
\label{eq:chi}
\end{equation}

As the Shepp-Olkin result (\ref{eq:maxEntB})
was only seen as an intermediate step in proving
the maximum entropy property of the Poisson 
distribution (\ref{eq:maxEntP}), we may still
hope that the corresponding result remains
true for compound Poisson measures,
namely that,
\begin{equation}
H(\mbox{CPo}(\lambda,Q))
=
\sup\Big\{
H(C_Q\bp{p})\;:\;
\vc{p}\in\Pc_\infty(\lambda)\Big\}.
\label{eq:maxEntPC}
\end{equation}
Again, (\ref{eq:maxEntPC}) fails in general. 
For example, taking the same 
$Q,\lambda$ and $\vc{p}$ as above,
yields,
$$
H(\mbox{CPo}(\lambda,Q))
<0.090765
<0.090804
< H(C_Q\bp{p}).$$


The main purpose of the present work
is to show that, despite these negative
results, it is possible to provide
natural, broad sufficient conditions,
under which the compound binomial and 
compound Poisson distributions can be
shown to have maximal entropy in an
appropriate class of measures.

Our first result (a more general version of which 
is proved in Section~3) states that,
as long as $Q$ and the compound Poisson measure
$\mbox{CPo}(\lambda,Q)$ are log-concave, 
the maximum entropy statement analogous to  
(\ref{eq:maxEntPJ}) remains valid 
in the compound Poisson case:

\begin{theorem} \label{thm:mainpoi}
If the distribution $Q$ on $\Nat$ and the compound Poisson distribution
$\mbox{\em CPo}(\lambda,Q)$ are both log-concave, 
then,
$$ 
H(\mbox{\em CPo}(\lambda,Q))
=\max\Big\{
H(C_Q P) \;:\; \mbox{ultra-log-concave $P$ with mean $\lambda$}
\Big\}.$$
\end{theorem}

The notion of log-concavity is central in the development
of the ideas in this work. Recall that 
the distribution $P$ of a random variable $X$ on $\Z_+$
is {\em log-concave} if its support is a (possibly infinite)
interval of successive integers in $\Z_+$, and,
\begin{equation} \label{eq:lcdef}
P(x)^2 \geq P(x+1) P(x-1),\;\;\;\; \mbox{for all $x \geq 1$.}
\end{equation}
We also recall that most of the 
commonly used distributions appearing 
in applications (e.g.,
the Poisson, binomial, geometric, negative binomial, hypergeometric
logarithmic series, or Polya-Eggenberger distribution)
are log-concave.

Note that ultra-log-concavity of $P$,
defined as in equation (\ref{eq:ulcdef}),
is more restrictive than log-concavity,
and it is equivalent to the requirement that 
ratio $P/\Pi_{\lambda}$ is a log-concave
sequence for some (hence all) $\lambda>0$.


Our second result  
states that (\ref{eq:maxEntBC})
{\em does} hold, under certain
conditions on $Q$ and CBin($n,\lambda,Q$):

\begin{theorem} \label{thm:mainber}
If the distribution $Q$ on $\Nat$ 
and the compound binomial distribution 
$\mbox{\em CBin}(n,\lambda/n,Q)$
are both log-concave, 
then,
$$H(\mbox{\em CBin}(n,\lambda/n,Q))
=\max\Big\{ H(C_Q\bp{p})\;:\; {\vc{p}\in\Pc_n(\lambda)}\Big\},$$
as long as the tail of $Q$ satisfies
either one of the following properties:
$(a)$~$Q$ has finite support; or 
$(b)$~$Q$ has tails heavy enough so that,
for some $\rho,\beta>0$ and $N_0\geq 1$, 
we have, $Q(x)\geq \rho^{x^\beta}$,
for all $x\geq N_0$.
\end{theorem}

The proof of Theorem~\ref{thm:mainber}  is given in 
Section~\ref{sec:compbin}.
As can be seen there,
conditions $(a)$ and $(b)$ are introduced 
purely for technical reasons, and can probably
be significantly relaxed; see Section~7 for
a further discussion.


It remains an open question to give {\em necessary
and sufficient} conditions on $\lambda$ and $Q$ for the compound 
Poisson and compound binomial distributions to have maximal 
entropy within an appropriately defined class. As a first step,
one may ask for natural conditions that imply that a compound binomial
or compound Poisson distribution is log-concave. We discuss several 
such conditions in Section~5. 

In particular, the discussion in Section~5 implies the following 
explicit maximum entropy statements.

\begin{example}
\begin{enumerate}
\item
Let $Q$ be an arbitrary log-concave distribution
on ${\mathbb N}$. Then Lemma~\ref{lem:lc} 
combined with Theorem~\ref{thm:mainber} implies that
the maximum entropy property of the
compound binomial distribution in
equation~{\em (\ref{eq:maxEntBC})} holds,
for all $\lambda$ large enough. That is,
the compound binomial 
{\em CBin($n,\lambda/n,Q$)} has maximal entropy
among all compound Bernoulli sums $C_Q\bp{p}$
with $p_1+p_2+\cdots+p_n=\lambda$, as long
as $\lambda \geq \frac{nQ(2)}{Q(1)^2+Q(2)}$.
\item
Let $Q$ be an arbitrary log-concave distribution
on ${\mathbb N}$. Then Theorem~\ref{thm:lcconj} 
combined with Theorem~\ref{thm:mainpoi} implies that
the maximum entropy property of the compound Poisson 
$CPo(\lambda,Q)$ holds if and only if
$\lambda\geq\frac{2Q(2)}{Q(1)^2}$.
\end{enumerate}
\end{example}

As mentioned in the introduction, the above results
can be used in order to gain better understanding
of ultra-log-concave sequences in combinatorics.
Specifically, as discussed in more detail in 
Section~6, they can be used to estimate how 
``spread out'' these sequences in terms of the
entropy.

\section{Maximum Entropy Property of the Compound Poisson Distribution} 
\label{sec:comppoi}

Here we show that, if $Q$ and the
compound Poisson distribution 
$\mbox{CPo}(\lambda,Q)=C_Q\Pi_\lambda$ 
are both log-concave, then 
$\mbox{CPo}(\lambda,Q)$
has maximum entropy among all 
distributions of the form $\cP$, when $P$ has mean 
$\lambda$ and is ultra-log-concave.
Our approach is an extension of the 
`semigroup' arguments of \cite{johnson21}.

We begin by recording some basic properties
of log-concave and ultra-log-concave distributions:
\begin{itemize}
\item[$(i)$]
If $P$ is ultra-log-concave, then
from the definitions it is immediate
that $P$ is log-concave.
\item[$(ii)$]
If $Q$ is log-concave, then it has finite moments
of all orders; see \cite[Theorem~7]{keilson}.
\item[$(iii)$]
If $X$ is a random variable
with ultra-log-concave distribution $P$, then (by~$(i)$ 
and~$(ii)$) it has finite moments of all orders.
Moreover, considering the covariance between the decreasing 
function $P(x+1) (x+1)/P(x)$ and the increasing function
$x(x-1) \cdots (x-n)$, shows that the falling 
factorial moments 
of $P$ satisfy, 
$$E[(X)_n]:=E[X(X-1) \cdots (X-n+1)] \leq (E(X))^n;$$ 
see \cite{johnson21} and \cite{johnsonc2}
for details. 
\item[$(iv)$] 
The Poisson distribution and all Bernoulli
sums are ultra-log-concave.
\end{itemize}

Recall the following definition 
from \cite{johnson21}:

\begin{definition} 
\label{def:stmap} 
Given $\alpha\in[0,1]$ and a random variable $X\sim P$ 
on $\Z_+$ with mean $\lambda\geq 0$,
let $U_\alpha P$ denote the 
distribution of the random variable,
$$\sum_{i=1}^X B_i+Z_{\lambda(1-\alpha)},$$
where the $B_i$ are i.i.d.\ $\bern{\alpha}$,
$Z_{\lambda(1-\alpha)}$ has distribution $\mbox{\em Po}(\lambda(1-\alpha))$,
and all random variables are independent
of each other and of $X$.
\end{definition}

Note that, if $X\sim P$ has mean $\lambda$,
then $U_\alpha P$ has the same mean. Also,
recall the following useful relation that
was established in 
Proposition~3.6 of \cite{johnson21}: For all $y\geq 0$,
\begin{equation} \label{eq:newheat}
\frac{\partial }{\partial \alpha} U_{\alpha}P(y)  = \frac{1}{\alpha} \left(
\lambda
(U_{\alpha}P(y) - U_{\alpha}P(y-1)
- ((y+1) U_{\alpha}P(y+1) - y U_{\alpha}P(y)) \right).
\end{equation}
%
Next we define another transformation of 
probability distributions $P$ on ${\mathbb Z}_+$: 

\begin{definition} \label{def:clustermap}
Given $\alpha\in[0,1]$, a distribution $P$ on $\Z_+$
and a compounding distribution $Q$ on $\Nat$, 
let $U^Q_{\alpha}P$ denote the
distribution $C_Q U_\alpha P$:
$$U^Q_{\alpha} P(x):=C_QU_\alpha P(x)
= \sum_{y=0}^{\infty} U_\alpha P(y) Q^{*y}(x),
\;\;\;\;x\geq 0.$$
\end{definition}

Work of Chafa\"{i} \cite{chafai} suggests that the semigroup 
of Definition \ref{def:stmap}
may be viewed as  the action of the $M/M/\infty$ queue. Similarly
the semigroup of Definition \ref{def:clustermap} corresponds to  the marginal 
distributions of a continuous-time hidden Markov process,  where the underlying 
Markov process is the $M/M/\infty$ queue and the  output at each time is 
obtained by a compounding operation.

\begin{definition} \label{def:sizebias}
For a distribution $P$ on $\Z_+$  with mean $\nu$,  
its {\em size-biased} distribution $P^{\#}$ on $\Z_+$  is defined by
$$
P^{\#}(y) = \frac{(y+1)P(y+1)}{\nu} .
$$
\end{definition}

An important observation that will be at the
heart of the proof of Theorem~\ref{thm:mainpoi}
below is that, for $\alpha=0$, $U_0^QP$
is simply the compound Poisson
measure CP$(\lambda,Q)$, while for $\alpha=1$,
$U_1^QP=C_QP$. The following lemma 
gives a rough bound on the third
moment of $U_\alpha^QP$:

\begin{lemma} \label{lem:moments}
Suppose $P$ is an ultra-log-concave 
distribution with mean $\lambda>0$ 
on ${\mathbb Z}_+$, and
let $Q$ be a log-concave compounding 
distribution on ${\mathbb N}$.
For each $\alpha\in[0,1]$,
let $W_\alpha,V_\alpha$ be random variables
with distributions $U_\alpha^QP=C_Q U_\alpha P$
and $C_Q (U_\alpha P)^{\#}$, respectively.
Then the third moments 
$E(W_\alpha^3)$
and $E(V_\alpha^3)$ are both bounded above by,
$$\lambda q_3 +3\lambda^2q_1q_2+\lambda^3q_1^3,$$
where $q_1,q_2,q_3$ denote the first, second and third
moments of $Q$, respectively.
\end{lemma}

\begin{proof}
Recall that, as stated in properties $(ii)$ and~$(iii)$ 
in the beginning of Section~\ref{sec:comppoi},
$Q$ has finite moments of all orders, and 
that the $n$th falling factorial moment
of any ultra-log-concave random variable $Y$ 
with distribution $R$ on ${\mathbb Z}_+$ is
bounded above by $(E(Y))^n$. Now for an arbitrary
ultra-log-concave distribution $R$, define
random variables $Y\sim R$ and $Z\sim C_Q R$.
If $r_1,r_2,r_3$ denote the first three moments
of $Y\sim R$, then,
\begin{eqnarray} 
E(Z^3)
&=&
	q_3r_1 + 3q_1q_2 E[(Y)_2] + q_1^3 E[(Y)_3]
	\nonumber\\
&\leq&
	q_3r_1 + 3q_1q_2r_1^2 + q_1^3r_1^3.
	\label{eq:3rdmomid}
\end{eqnarray}
Since the map $U_\alpha$ preserves ultra-log-concavity
\cite{johnson21}, if $P$ is ultra-log-concave then 
so is $R = U_{\alpha} P$, so that (\ref{eq:3rdmomid}) 
gives the required bound for the third moment of
$W_\alpha$, upon noting that the mean of
the distribution $U_\alpha P$ is equal to $\lambda$.

Similarly, size-biasing preserves ultra-log-concavity; 
that is, if $R$ is ultra-log-concave, then so is $R^{\#}$, since 
$R^{\#}(x+1)(x+1)/R^{\#}(x) = (R(x+2) (x+2)(x+1))/(R(x+1) (x+1))
= R(x+2) (x+2)/R(x+1)$ is also decreasing.
Hence, $R'=(U_\alpha P)^{\#}$ is ultra-log-concave, 
and (\ref{eq:3rdmomid}) applies in this case as well.
In particular, noting that the mean of 
$Y'\sim R'= (U_\alpha P)^{\#}=R^{\#}$
can be bounded in terms of the mean of $Y\sim R$ as,
$$E(Y')=\sum_x x\frac{(x+1)U_\alpha P(x+1)}{\lambda}
=\frac{E[(Y)_2]}{E(Y)}\leq\frac{\lambda^2}{\lambda}=\lambda,$$
the bound (\ref{eq:3rdmomid}) yields
the required bound for the third 
moment of $V_\alpha$.
\end{proof}

In \cite{johnson21}, the characterization
of the Poisson as a maximum entropy 
distribution was proved through
the decrease of its score function. In
an analogous way, 
we define the score function 
of a $Q$-compound random variable as follows,
cf.\ \cite{johnson22}, 

\begin{definition} \label{def:score}
Given a distribution $P$ on $\Z_+$ with mean $\lambda$, 
the corresponding $Q$-compound distribution
$\cP$ has score function defined by,
$$
\sco{\cP}(x)=\frac{C_Q(P^{\#})(x)}{C_QP(x)}-1,\;\;   x\geq 0 .
$$
\end{definition}

More explicitly, one can write
\begin{equation} \label{eq:score}
\sco{\cP}(x) = \frac{ \sum_{y=0}^{\infty} (y+1) P(y+1) \qst{y}(x)}{\lambda
\sum_{y=0}^{\infty} P(y) \qst{y}(x) } - 1 .
\end{equation}

Notice that the mean of 
of $\sco{\cP}$ with respect to $\cP$ is zero,
and that if $P\sim\mbox{Po}(\lambda)$
then $\sco{\cP}(x) \equiv 0$. Further, 
when $Q$ is the point mass at 1
this score function 
reduces to the ``scaled score function'' introduced in \cite{johnson11}.
But, unlike the scaled score function and 
an alternative score function given in
\cite{johnson22}, this score function is not only a function 
of the compound distribution $\cP$, but also explicitly depends
on $P$. A projection identity and other properties of 
$\sco{\cP}$ are proved in \cite{johnson22}.

Next we show that, if $Q$ is log-concave and $P$
is ultra-log-concave, then the score function 
$\sco{\cP}(x)$ is decreasing in $x$.

\begin{lemma} \label{lem:decsc}
If $P$ is ultra-log-concave and the compounding 
distribution $Q$ is log-concave,
then the score function $\sco{\cP}(x)$ of 
$\cP$ is decreasing in $x$. 
\end{lemma}
\begin{proof} First we recall Theorem~2.1 of Keilson 
and Sumita \cite{keilson2}, 
which implies that,
if $Q$ is log-concave, then for any $m \geq n$, and for any $x$:
\begin{equation} \label{eq:tech}
\qst{m}(x+1) \qst{n}(x) - \qst{m}(x) \qst{n}(x+1) \geq 0. \end{equation}
[This can be proved by
considering $\qst{m}$ as the convolution of $\qst{n}$ and $\qst{(m-n)}$, 
and writing
\begin{eqnarray*}
\lefteqn{ \qst{m}(x+1) \qst{n}(x) - \qst{m}(x) \qst{n}(x+1)  } \\
& = & \sum_l \qst{(m-n)}(l) \bigg( \qst{n}(x+1-l) \qst{n}(x) -
\qst{n}(x-l) \qst{n}(x+1) \bigg).
\end{eqnarray*}
Since $Q$ is log-concave, then 
so is $\qst{n}$, 
cf.\ \cite{karlin3},
so the ratio $\qst{n}(x+1)/\qst{n}(x)$ is decreasing in $x$, and  
(\ref{eq:tech}) follows.]

By definition, $\sco{\cP}(x) \geq \sco{\cP}(x+1)$ if and only if,
\begin{eqnarray}
0 & \leq & \left( \sum_y (y+1) P(y+1) \qst{y}(x) \right) \left(
\sum_z P(z) \qst{z}(x+1) \right) \nonumber \\
& &  - \left( \sum_y (y+1) P(y+1) \qst{y}(x+1) \right) 
\left( \sum_z P(z) \qst{z}(x) \right) \nonumber \\
& = & \sum_{y,z} (y+1) P(y+1) P(z) \left[ \qst{y}(x) \qst{z}(x+1)
- \qst{y}(x+1) \qst{z}(x) \right]. \label{eq:doublesum}
\end{eqnarray}
Noting that for $y=z$ the term in square brackets in the
double sum becomes zero, and swapping the values of $y$ and
$z$ in the range $y>z$,
the double sum in
(\ref{eq:doublesum}) becomes,
$$ \sum_{y < z} [(y+1) P(y+1) P(z) - (z+1) P(z+1) P(y)]
\left[ \qst{y}(x) \qst{z}(x+1)
- \qst{y}(x+1) \qst{z}(x) \right].$$
By the ultra-log-concavity of $P$, the first square 
bracket is positive for $y \leq z$,
and by equation~(\ref{eq:tech}) the second square bracket is 
also positive for $y \leq z$.
\end{proof}

We remark that, under the same assumptions, and using a very similar 
argument, an analogous result holds for some alternative
score functions recently introduced in \cite{johnson22}
and in related work.

Combining Lemmas~\ref{lem:decsc} and~\ref{lem:moments}
with equation~(\ref{eq:newheat}) 
we deduce the following result,
which is the main technical step
in the proof of Theorem~\ref{thm:mainpoi} below.

\begin{proposition} \label{prop:deriv}
Let $P$ be an ultra-log-concave distribution on ${\mathbb Z}_+$
with mean $\lambda>0$, and assume that 
$Q$ and $\mbox{\em CPo}(\lambda,Q)$ are
both log-concave. Let $W_\alpha$ be a 
random variable with  distribution $U_\alpha^QP$, 
and define, for all $\alpha\in[0,1],$
the function,
$$E(\alpha):=E[-\log \CPi_{\lambda}(W_{\alpha})].$$
Then $E(\alpha)$ is continuous for all $\alpha\in[0,1]$,
it is differentiable for $\alpha\in(0,1)$, and,
moreover, $E'(\alpha)\leq 0$ for $\alpha\in(0,1)$.
In particular, $E(0)\geq E(1)$.
\end{proposition}

\begin{proof} 
Recall that, 
$$ 
U^Q_{\alpha}P(x) = 
C_Q U_\alpha P(x)
= 
\sum_{y=0}^{\infty} U_{\alpha}P(y) \qst{y}(x)
=\sum_{y=0}^{x} U_{\alpha}P(y) \qst{y}(x),
$$
where the last sum is restricted to the
range $0\leq y\leq x$, because
$Q$ is supported on $\Nat$.
Therefore, since $U_\alpha P(x)$ is continuous
in $\alpha$ \cite{johnson21},
so is $U_\alpha^Q P(x)$,
and to show that $E(\alpha)$ is continuous
it suffices to show that the series,
\begin{eqnarray}
E(\alpha)
:=E[-\log \CPi_{\lambda}(W_{\alpha})]
=-\sum_{x=0}^\infty U_\alpha^QP(x)\log\CPi_\lambda(x),
\label{eq:series}
\end{eqnarray}
converges uniformly. To that end,
first observe that log-concavity of $C_Q\Pi_\lambda$
implies that $Q(1)$ is nonzero. [Otherwise,
if $i>1$ be the smallest integer $i$ such that $Q(i)\neq 0$, 
then $C_Q\Pi_\lambda(i+1)=0$, but 
$C_Q\Pi_\lambda(i)$ and
$C_Q\Pi_\lambda(2i)$ are both strictly positive,
contradicting the log-concavity of 
$C_Q\Pi_\lambda$.]
Since $Q(1)$ is nonzero, we can bound the
compound Poisson probabilities as, 
$$1 \geq \CPi_{\lambda}(x) = \sum_{y} [e^{-\lambda} \lambda^y/y!]\qst{y}(x)
\geq e^{-\lambda} [\lambda^x/x!] Q(1)^x, 
\;\;\;\;\mbox{for all}\;x\geq 1,$$
so that the summands in (\ref{eq:series})
can be bounded,
\begin{equation} \label{eq:boundlog}
0 \leq - \log \CPi_{\lambda}(x) \leq \lambda + \log x! - x \log( \lambda Q(1))
\leq Cx^2,
\;\;\;\;x\geq 1,\end{equation} 
for a constant $C>0$ that depends only on $\lambda$ and $Q(1)$.
Therefore, for any $N\geq 1$, the tail of the series (\ref{eq:series})
can be bounded,
$$
0\leq -\sum_{x=N}^\infty U_\alpha^QP(x)\log\CPi_\lambda(x)
\leq C E[W^2_\alpha{\mathbb I}_{\{W_\alpha\geq N\}}]
\leq \frac{C}{N}E[W_\alpha^3],$$
and, in view of Lemma~\ref{lem:moments}, 
it converges uniformly.

Therefore, $E(\alpha)$ is continuous in $\alpha$, 
and, in particular, convergent for all $\alpha\in[0,1]$.
To prove that it is differentiable at each $\alpha\in(0,1)$
we need to establish that: (i)~the summands in (\ref{eq:series})
are continuously differentiable in $\alpha$ for each $x$; 
and (ii)~the series
of derivatives converges uniformly. 

Since, as noted above, $U_\alpha^Q P(x)$ is defined
by a finite sum, we can differentiate with respect 
to $\alpha$ under the sum, to obtain,
\begin{eqnarray}
\frac{\partial}{\partial \alpha} 
U^Q_{\alpha} P(x)
=
\frac{\partial}{\partial \alpha} 
C_Q U_\alpha P(x)
= \sum_{y=0}^{x} \frac{\partial}{\partial \alpha} 
U_{\alpha}P(y) \qst{y}(x).
\label{eq:finite}
\end{eqnarray}
And since $U_\alpha P$ is continuously 
differentiable in $\alpha\in(0,1)$
for each $x$ (cf.\ \cite[Proposition~3.6]{johnson21}
or equation (\ref{eq:newheat}) above),
so are the summands in (\ref{eq:series}),
establishing~(i); in fact, they are
infinitely differentiable, which can be seen 
by repeated applications of (\ref{eq:newheat}). 
To show that the
series of derivatives converges uniformly,
let $\alpha$ be restricted in an arbitrary
open interval $(\epsilon,1)$ for some $\epsilon>0$.
The relation (\ref{eq:newheat})
combined with (\ref{eq:finite}) yields,
for any $x$,
\begin{eqnarray}
\lefteqn{\frac{\partial}{\partial \alpha} U_\alpha^Q P (x)} \nonumber \\
& = & \sum_{y=0}^{x} 
\biggl( \lambda
(U_{\alpha}P(y) - U_{\alpha}P(y-1)
- ((y+1) U_{\alpha}P(y+1) - y U_{\alpha}P(y)) \biggr)
\qst{y}(x) \nonumber \\
& = & -\frac{1}{\alpha}  \sum_{y=0}^{x} 
\left( (y+1) U_{\alpha}P(y+1) - \lambda U_{\alpha}P(y) \right)
(\qst{y}(x) - \qst{y+1}(x)) \nonumber \\
& = & - \frac{1}{\alpha} \sum_{y=0}^{x} 
\left( (y+1) U_{\alpha}P(y+1) - \lambda U_{\alpha}P(y) \right)
\qst{y}(x) \nonumber \\
& &  + \sum_{v=0}^{x} Q(v) \frac{1}{\alpha} \sum_{y=0}^{x} 
\left( (y+1) U_{\alpha}P(y+1) - \lambda U_{\alpha}P(y) \right)
\qst{y}(x-v) \nonumber \\
& = & - \frac{\lambda}{\alpha} U_\alpha^Q P(x) \left( 
\frac{ \sum_{y=0}^{x} (y+1) U_{\alpha}P(y+1) \qst{y}(x)}{\lambda
U_\alpha^Q P(x)  } - 1 \right) \nonumber \\
&  &  + \frac{\lambda}{\alpha}
\sum_{v=0}^{x} Q(v) U_\alpha^Q P(x-v) \left( 
\frac{ \sum_{y=0}^{x} (y+1) U_{\alpha}P(y+1) \qst{y}(x-v)}{\lambda
U_\alpha^Q P(x-v)  } - 1 \right) \nonumber \\
& = & - \frac{\lambda}{\alpha} \left( U_\alpha^Q P(x) \sco{U_\alpha^Q P}(x)
- \sum_{v=0}^{x} Q(v) U_\alpha^Q P(x-v) \sco{U_\alpha^Q P}(x-v) \right) .
\label{eq:derivative} \end{eqnarray}
Also,
for any $x$, by definition,
$$|U_\alpha^Q P(x) \sco{U_\alpha^Q P}(x)| 
\leq 
C_Q(U_\alpha P)^{\#}(x) +U_\alpha^QP(x), 
$$ 
where, for any distribution $P$, we write
$P^{\#}(y) = P(y+1)(y+1)/\lambda$ for its size-biased version.
Hence for any $N\geq 1$, equations
(\ref{eq:derivative}) and (\ref{eq:boundlog}) yield the bound,
\begin{eqnarray*}
\lefteqn{
\left| \sum_{x=N}^{\infty} \frac{\partial}{\partial \alpha} 
	U_\alpha^Q P(x) \log \CPi_{\lambda}(x) \right| } \\
& \leq & 
	\sum_{x=N}^{\infty} 
	\frac{C\lambda x^2}{\alpha}
	\Big\{ 
	C_Q(U_\alpha P)^{\#}(x) +U_\alpha^QP(x)
	+ \sum_{v=0}^{x} Q(v)
	[
	C_Q(U_\alpha P)^{\#}(x-v) +U_\alpha^QP(x-v)
	] \Big\}\\
& = & 
	\frac{2C}{\alpha}
	E\Big[
	\Big(
	V_\alpha^2+W_\alpha^2+X^2
	+XV_\alpha+XW_\alpha
	\Big)
	{\mathbb I}_{\{V_\alpha\geq N,\;W_\alpha\geq N,\;X\geq N\}}
	\Big]\\
&\leq&
	\frac{C'}{\alpha}
	\Big\{
	E[V_\alpha^2 {\mathbb I}_{\{V_\alpha\geq N\}}]
	+E[W_\alpha^2 {\mathbb I}_{\{W_\alpha\geq N\}}]
	+E[X^2 {\mathbb I}_{\{X\geq N\}}]
	\Big\}\\
&\leq&
	\frac{C'}{N\alpha}
	\Big\{
	E[V_\alpha^3]
	+E[W_\alpha^3]
	+E[X^3]
	\Big\},
\end{eqnarray*}
where $C,C'>0$ are appropriate finite constants, 
and the random variables 
$V_\alpha\sim C_Q(U_\alpha P)^{\#}$,
$W_\alpha\sim U^Q_\alpha P$ and $X\sim Q$ are independent.
Lemma~\ref{lem:moments} implies that this bound
converges to zero uniformly in $\alpha\in(\epsilon,1)$, as $N\to\infty$.
Since $\epsilon>0$ was arbitrary,
this establishes that 
$E(\alpha)$ is differentiable for all $\alpha\in(0,1)$
and, in fact, that we can differentiate the 
series (\ref{eq:series})
term-by-term, to obtain,
\begin{eqnarray}
\lefteqn{ 
E'(\alpha) 
\;=\;
	- \sum_{x=0}^{\infty} \frac{\partial}{\partial \alpha} 
	U_\alpha^Q P(x) \log \CPi_{\lambda}(x)
	} 
	\label{eq:step1}  \\
& = & 
	\frac{\lambda}{\alpha} \sum_{x=0}^{\infty}
	\left( U_\alpha^Q P(x) \sco{U_\alpha^Q P}(x)
	- \sum_{v=0}^{x} Q(v) U_\alpha^Q P(x-v) \sco{U_\alpha^Q P}(x-v) \right)
	\log \CPi_{\lambda}(x) \nonumber \\
& = & 
	\frac{\lambda}{\alpha} \sum_{x=0}^{\infty} 
	U_\alpha^Q P(x) \sco{U_\alpha^Q P}(x)
	\left( \log \CPi_{\lambda}(x) - \sum_{v=0}^{\infty} Q(v) 
	\log \CPi_{\lambda}(x+v) 
	\right),
	\nonumber
\end{eqnarray}
where the second equality follows from using
(\ref{eq:derivative}) above, and the rearrangement 
leading to the third equality follows by interchanging
the order of (second) double summation and replacing $x$ 
by $x+v$.

Now we note that, exactly as in \cite{johnson21}, 
the last series above is the covariance between
the (zero-mean) function $\sco{U_\alpha^Q P}(x)$ 
and the function $\left( \log \CPi_{\lambda}(x) 
- \sum_v Q(v) \log \CPi_{\lambda}(x+v) \right)$,
under the measure $U_\alpha^Q P$.
Since $P$ is ultra-log-concave, so is $U_\alpha P$
\cite{johnson21}, hence the score function
$\sco{U_\alpha^Q P}(x)$ is decreasing in $x$, 
by Lemma~\ref{lem:decsc}. Also, the 
log-concavity of $\CPi_{\lambda}$ implies that the 
second function is increasing, and
Chebyshev's rearrangement lemma
implies that the covariance is 
less than or equal to zero, proving
that $E'(\alpha)\leq 0$, as claimed.

Finally, the fact that $E(0)\geq E(1)$
is an immediate consequence of the
continuity of $E(\alpha)$ on $[0,1]$
and the fact that 
$E'(\alpha)\leq 0$ for all $\alpha\in(0,1)$.
\end{proof}

Notice that, for the above proof to work, it is not necessary that
$\CPi_{\lambda}$ be log-concave; the weaker property
that $\left( \log \CPi_{\lambda}(x) 
- \sum_v Q(v) \log \CPi_{\lambda}(x+v) \right)$ be increasing is enough.

We can now state and prove a slightly more general form of Theorem~\ref{thm:mainpoi}.
Recall that the relative entropy between distributions $P$ and $Q$ on $\Z_+$,  denoted by $D(P\|Q)$, 
is defined by
$$D(P\|Q):=\sum_{x\geq 0}P(x)\log\frac{P(x)}{Q(x)}.$$

\begin{theorem} \label{thm:mainpoi-D}
Let $P$ be an ultra-log-concave distribution on $\Z_+$ with mean $\lambda$.
If the distribution $Q$ on $\Nat$ and the compound Poisson distribution
$\CPi_\lambda$ are both log-concave, then
$$ 
D(C_Q P\|\CPi_\lambda) \leq
H(\CPi_\lambda)- H(C_Q P) .
$$
\end{theorem}

\begin{proof}
As in Proposition~\ref{prop:deriv},
let $W_\alpha\sim U^Q_\alpha P=C_Q U_\alpha P$.
Noting that $W_0\sim \CPi_\lambda$ and $W_1\sim C_Q P$,
we have
\begin{eqnarray*}
	H(C_Q P)+D(C_Q P\|\CPi_\lambda)
&=&
	-E[\log \CPi_\lambda(W_1)]\\
&\leq&
	-E[\log \CPi_\lambda(W_0)]\\
&=&
	H(\CPi_\lambda),
\end{eqnarray*}
where the inequality is exactly
the statement that $E(1)\leq E(0)$,
proved in Proposition~\ref{prop:deriv}.
\end{proof}

Since $0\leq D(C_Q P\|\CPi_\lambda)$,
Theorem~\ref{thm:mainpoi} immediately follows.

\section{Maximum Entropy Property of the Compound Binomial Distribution} 
\label{sec:compbin}

Here we prove the maximum entropy result for compound 
binomial random variables, Theorem~\ref{thm:mainber}. 
The proof, to some extent, parallels some
of the arguments in \cite{harremoes}\cite{mateev}\cite{shepp},
which rely on differentiating the compound-sum probabilities
$\bp{p}(x)$ 
for a given parameter vector $\vc{p}=(p_1,p_2,\ldots,p_n)$
(recall Definition~\ref{def:bp} in the Introduction),
with respect to an individual $p_i$.
Using the representation,
\begin{equation} \label{eq:master}
\cp{p}(y) = 
\sum_{x=0}^n \bp{p}(x) Q^{*x}(y),
\;\;\;\;y\geq 0,
\end{equation}
differentiating $\cp{p}(x)$
reduces to differentiating
$\bp{p}(x)$,
and leads to an expression
equivalent to that derived
earlier 
in (\ref{eq:derivative})
for the derivative of $C_Q U_\alpha P$
with respect to $\alpha$.
\begin{lemma} \label{lem:partials}
Given a parameter vector $\vc{p}=(p_1,p_2,\ldots,p_n)$,
with $n\geq 2$ and
each $0 \leq p_i \leq 1$, 
let, 
$$ \vc{p_t} = \left( \frac{p_1 + p_2}{2} + t, \frac{p_1 + p_2}{2} - t, p_3, \ldots, p_n \right),$$
for $t \in [-(p_1+p_2)/2, (p_1 + p_2)/2]$. Then,
\begin{equation} \label{eq:maindiff}
\frac{\partial}{\partial t} \cp{p_t}(x) =  (- 2t) 
\sum_{y=0}^n \bp{\wt{p}}(y)
\left( Q^{*(y+2)}(x) - 2 Q^{*(y+1)}(x) + Q^{*y}(x) \right), 
\end{equation}
where $\vc{\wt{p}} = (p_3, \ldots, p_n)$.  
\end{lemma}
\begin{proof} Note that the sum of the entries of $\vc{p}_t$ is
constant as $t$ varies, and that $\vc{p_t} = \vc{p}$
for $t = (p_1 - p_2)/2$, while
$\vc{p_t} = ( (p_1 + p_2)/2, (p_1 + p_2)/2, p_3, \ldots, p_n)$
for $t=0$. Writing
$k = p_1 + p_2$, $\bp{p_t}$ can be expressed,
\begin{eqnarray*}
\bp{p_t}(y) & = & 
\left(  \frac{k^2}{4} - t^2 \right) 
\bp{\wt{p}}(y-2)   
+ \left( k \left( 1 - \frac{k}{2} \right)
+2 t^2 \right) \bp{\wt{p}}(y-1) \\
& & + \left( \left( 1 - \frac{k}{2} \right)^2 - t^2 \right) 
\bp{\wt{p}}(y),
\end{eqnarray*}
and its derivative with respect to $t$ is,
$$ \frac{ \partial}{\partial t} \bp{p_t}(y) 
= - 2t \left(  \bp{\wt{p}}(y-2) - 2 \bp{\wt{p}}(y-1) + \bp{\wt{p}}(y) \right).$$
The expression (\ref{eq:master}) for
$\cp{p_t}$ shows that it is
a finite linear combination of compound-sum
probabilities $\bp{p_t}(x)$,
so we can differentiate inside the sum to obtain, 
\begin{eqnarray*}
\frac{ \partial}{\partial t} \cp{p_t}(x)
&  = & \sum_{y=0}^n \frac{ \partial}{\partial t} \bp{p_t}(y)
Q^{*y}(x) \\
& = & - 2t \sum_{y=0}^n \left(  \bp{\wt{p}}(y-2) - 2 \bp{\wt{p}}(y-1) + \bp{\wt{p}}(y) \right) Q^{*y}(x) \\
& = & -2 t \sum_{y=0}^{n-2} \bp{\wt{p}}(y) \left( 
Q^{*(y+2)}(x) - 2 Q^{*(y+1)}(x) + Q^{*y}(x) \right),  
\end{eqnarray*}
since $\bp{\wt{p}}(y) = 0$ for $y \leq -1$ and $y \geq n-1$. \end{proof}

Next we state and prove the equivalent 
of Proposition~\ref{prop:deriv} above. Note
that 
the distribution of a compound Bernoulli sum
is invariant under permutations of the 
Bernoulli parameters $p_i$. Therefore,
the assumption $p_1\geq p_2$ is made below
without
loss of generality.
\begin{proposition} \label{prop:deriv2}
Suppose that the distribution $Q$ on $\Nat$ 
and the compound binomial distribution 
$\mbox{\em CBin}(n,\lambda/n,Q)$
are both log-concave; let 
$\vc{p}=(p_1,p_2,\ldots,p_n)$ be a 
given parameter vector with $n\geq 2$,
$p_1 +p_2+ \ldots + p_n = \lambda>0$,
and $p_1\geq p_2$;
let $W_t$ be a 
random variable with distribution $\cp{p_t}$;
and define, for all $t\in[0,(p_1-p_2)/2],$
the function,
$$E(t):=E[-\log \cp{\pbar}(W_t)],$$
where $\pbar$ denotes the parameter
vector with all entries equal to $\lambda/n$.
If $Q$ satisfies either of the conditions:
$(a)$~$Q$ finite support; or 
$(b)$~$Q$ has tails heavy enough so that,
for some $\rho,\beta>0$ and $N_0\geq 1$, 
we have, $Q(x)\geq \rho^{x^\beta}$,
for all $x\geq N_0$, then
$E(t)$ is continuous for all 
$t\in[0,(p_1-p_2)/2]$,
it is differentiable for 
$t\in(0,(p_1-p_2)/2)$,
and, moreover, $E'(t)\leq 0$ for
$t\in(0,(p_1-p_2)/2)$.
In particular, $E(0)\geq E((p_1-p_2)/2)$.
\end{proposition}

\begin{proof} 
The compound distribution $C_Q\bp{p_t}$ is
defined by the finite sum,
$$
C_Q\bp{p_t}(x)=\sum_{y=0}^n\bp{p_t}(y)Q^{*y}(x),$$
and is, therefore, continuous in $t$. First,
assume that $Q$ has finite support.
Then so does $C_Q\bp{p}$ for any parameter
vector $\vc{p}$, and the continuity and 
differentiability of $E(t)$ are trivial.
In particular, the series defining $E(t)$ 
is a finite sum, so we can differentiate 
term-by-term, to obtain,
\begin{eqnarray}
E'(t)
& = & - \sum_{x=0}^{\infty} \frac{\partial}{\partial t} \cp{p_t}(x) 
	\log \cp{\pbar}(x) \nonumber  \\
& = & 2t \sum_{x=0}^{\infty}
\sum_{y=0}^{n-2} \bp{\wt{p}}(y)
\left( Q^{*(y+2)}(x) - 2 Q^{*(y+1)}(x) + Q^{*y}(x) \right) 
\log \cp{\pbar}(x)
\label{eq:binstep2}  \\
& = & 2t \sum_{y=0}^{n-2} \sum_{z=0}^{\infty} \bp{\wt{p}}(y) Q^{*y}(z) \sum_{v,w} Q(v) Q(w)
\bigg[ \log \cp{\pbar}(z+v+w) - \log \cp{\pbar}(z+v)  \nonumber \\
& & \hspace*{6.5cm}
-  \log \cp{\pbar}(z+w) + \log \cp{\pbar}(z)
\bigg], \label{eq:binstep3}
\end{eqnarray}
where (\ref{eq:binstep2}) follows by Lemma~\ref{lem:partials}. 
By assumption, the distribution $\cp{\pbar}=\mbox{CBin}(n,\lambda/n,Q)$ 
is log-concave,
which implies that,
for all $z,v,w$ such that $z+v+w$ is in the
support of $\mbox{CBin}(n,\lambda/n,Q)$,
\begin{equation*} 
\frac{ \cp{\pbar}(z)}{\cp{\pbar}(z+v)}
\leq \frac{ \cp{\pbar}(z+w)}{\cp{\pbar}(z+v+w)}.
\end{equation*}
Hence the term in square brackets in equation (\ref{eq:binstep3}) 
is negative, and the result follows.

Now, suppose condition $(b)$ holds on the tails of $Q$.
First we note that the moments of $W_t$ are all uniformly
bounded in $t$: Indeed, for any $\gamma>0$,
\begin{equation}
E[W_t^\gamma]=\sum_{x=0}^\infty 
C_Q\bp{p_t}(x)
x^\gamma 
=
\sum_{x=0}^\infty 
\sum_{y=0}^n\bp{p_t}(y) Q^{*y}(x)
x^\gamma
\leq
\sum_{y=0}^n
\sum_{x=0}^\infty 
Q^{*y}(x)
x^\gamma
\leq C_nq_\gamma,
\label{eq:moment}
\end{equation}
where $C_n$ is a constant depending
only on $n$, and $q_\gamma$ is the
$\gamma$th moment of $Q$, which
is of course finite; recall property~$(ii)$
in the beginning of Section~\ref{sec:comppoi}.

For the continuity of $E(t)$, it suffices to show that 
the series,
\begin{eqnarray}
E(t):=E[-\log \cp{\pbar}(W_t)]=
-\sum_{x=0}^\infty C_Q\bp{p_t}(x)\log C_Q\bp{\pbar}(x),
\label{eq:Eseries}
\end{eqnarray}
converges uniformly. The tail assumption on $Q$ implies
that, for all $x\geq N_0$,
$$1 \geq \cp{\pbar}(x) = \sum_{y=0}^n \bp{\pbar}(y) \qst{y}(x)
\geq \lambda(1-\lambda/n)^{n-1} Q(x)
\geq \lambda(1-\lambda/n)^{n-1} \rho^{x^\beta},$$
so that,
\begin{equation}
0\leq -\log \cp{\pbar}(x)\leq Cx^\beta,
\label{eq:logQ}
\end{equation}
for an appropriate constant $C>0$.
Then, for $N\geq N_0$, the tail of the series
(\ref{eq:Eseries}) can be bounded,
$$0\leq -\sum_{x=N}^\infty C_Q\bp{p_t}(x)\log C_Q\bp{\pbar}(x)
\leq C E[ W^\beta_t{\mathbb I}_{\{W_t\geq N\}}]
\leq \frac{C}{N}E[W_t^{\beta+1}]
\leq \frac{C}{N}C_nq_{\beta+1},
$$
where the last inequality follows from (\ref{eq:moment}).
This obviously converges to zero, uniformly
in $t$, therefore $E(t)$ is continuous.

For the differentiability of $E(t)$, 
note that the summands in (\ref{eq:series}) are
continuously differentiable (by Lemma~\ref{lem:partials}),
and that the series of derivatives converges uniformly
in $t$; to see that, for $N\geq N_0$ we apply
Lemma~\ref{lem:partials} together with the bound
(\ref{eq:logQ}) to get,
\begin{eqnarray*}
\lefteqn{
	\left| \sum_{x=N}^{\infty} \frac{\partial}{\partial t} 
	\cp{p_t}(x) \log \cp{\pbar}(x) \right| 
	} \\
& \leq & 
	2 t \sum_{x=N}^{\infty} 
	\sum_{y=0}^n \bp{\wt{p}}(y)
	\left( Q^{*(y+2)}(x) + 2 Q^{*(y+1)}(x) + Q^{*y}(x) \right) 
	Cx^\beta\\
& \leq & 
	2 C t 
	\sum_{y=0}^n 
	\sum_{x=N}^{\infty} 
	\left( Q^{*(y+2)}(x) + 2 Q^{*(y+1)}(x) + Q^{*y}(x) \right) 
	x^\beta,
\end{eqnarray*}
which is again easily seen to converge to zero
uniformly in $t$ as $N\to\infty$, since
$Q$ has finite moments of all orders.
This establishes the differentiability of $E(t)$
and justifies the term-by-term differentiation
of the series (\ref{eq:series}); the rest of 
the proof that $E'(t)\leq 0$ is the same as in case~$(a)$.
\end{proof}

Note that, as with Proposition~\ref{prop:deriv}, 
the above proof only requires that the compound
binomial distribution $\mbox{CBin}(n,\lambda/n,Q)=\cp{\pbar}$ 
satisfies a property weaker than log-concavity, namely 
that the function,
$\log \cp{\pbar}(x) - \sum_v Q(v) \log \cp{\pbar}(x+v),$
be increasing in $x$.

\begin{proof}{\bf (of Theorem~\ref{thm:mainber})}
Assume, without loss of generality,
that $n\geq 2$. If $p_1 > p_2$, then 
Proposition~\ref{prop:deriv2}
says that, $E((p_1-p_2)/2)\leq E(0)$, that is,
$$ -\sum_{x=0}^{\infty} 
\cp{p}(x) \log \cp{\pbar}(x)
\leq - \sum_{x=0}^{\infty} 
\cp{p^*}(x) \log \cp{\pbar}(x),$$
where $\vc{p^*} = ((p_1 + p_2)/2, (p_1 + p_2)/2, p_3, \ldots p_n)$
and $\vc{\pbar} = (\lambda/n,\ldots,\lambda/n)$.
Since the expression in the above right-hand-side
is invariant under permutations of the elements of 
the parameter vectors,
we deduce that it is maximized
by $\vc{p_t} = \pbar$. Therefore,
using, as before, the nonnegativity
of the relative entropy,
\begin{eqnarray*}
H(\cp{p})
&\leq&
	H(\cp{p}) + D(\cp{p} \| \cp{\pbar})\\
& = &
	-\sum_{x=0}^{\infty} 
	\cp{p}(x) \log \cp{\pbar}(x) \\
&\leq&
	-\sum_{x=0}^{\infty}
	\cp{\pbar}(x) \log \cp{\pbar}(x)\\
& = & 
H( \cp{\pbar} )
\;=\; H(\mbox{CBin}(n,\lambda/n,Q)),
\end{eqnarray*}
as claimed.
\end{proof}

Clearly one can also state a slightly more general version of 
Theorem~\ref{thm:mainber} analogous to Theorem~\ref{thm:mainpoi-D}.

\section{Conditions for Log-Concavity} \label{sec:lccond}

Theorems~\ref{thm:mainpoi} and~\ref{thm:mainber} state that 
log-concavity is a sufficient condition for 
compound binomial and compound Poisson distributions 
to have maximal entropy within a natural class. 
In this section, we discuss when
log-concavity holds.

Recall that Steutel and van Harn \cite[Theorem~2.3]{steutel2} showed that,
if $\{i Q(i)\}$ is a decreasing sequence, then 
CPo$(\lambda,Q)$ is a unimodal distribution,
which is a necessary condition for log-concavity.
Interestingly, the same condition 
provides a dichotomy of results
in compound Poisson approximation
bounds as developed by Barbour, Chen and Loh \cite{barbour-chen-loh}:
If $\{i Q(i)\}$ is decreasing,
then the bounds are
of the same form and order as in
the Poisson case, otherwise 
the bounds are much larger.
In a slightly different direction, Cai and Willmot \cite[Theorem~3.2]{cai} 
showed that if $\{Q(i)\}$ is decreasing
then the cumulative distribution function of the compound Poisson 
distribution CPo$(\lambda,Q)$, evaluated at the integers, 
is log-concave.
Finally, Keilson and Sumita \cite[Theorem~4.9]{keilson2} proved
that, if $Q$ is log-concave, then 
the ratio,
$$ \frac{ \CPi_{\lambda}(n)}{\CPi_{\lambda}(n+1)}, \;\;
$$
is decreasing in $\lambda$ for any fixed $n$.

In the present context, we first show
that a compound Bernoulli sum  is log-concave 
if the compounding distribution $Q$ is log-concave
and the Bernoulli parameters are sufficiently large.

\begin{lemma} \label{lem:lc}
Suppose $Q$ is a log-concave distribution on $\Nat$,
and all the elements $p_i$
of the parameter vector $\vc{p}=(p_1,p_2,\ldots,p_n)$
satisfy $p_i \geq \frac{1}{1 + Q(1)^2/Q(2)}$.
Then the compound Bernoulli sum distribution
$\cp{p}$ is log-concave. 
\end{lemma}

\begin{proof} 
Observe that, given that $Q$ is log-concave, 
the compound Bernoulli distribution
$\cbern{p}{Q}$ is log-concave if and only if,
\begin{equation}\label{cbern-lc}
p \geq \frac{1}{1 + Q(1)^2/Q(2)}.
\end{equation} 
Indeed, let $Y$ have distribution
$\cbern{p}{Q}$. Since $Q$ is log-concave
itself, the log-concavity of
$\cbern{p}{Q}$ is equivalent to 
the inequality,
$ \Pr(Y=1)^2 \geq \Pr(Y=2) \Pr(Y=0)$, 
which states that,
$ (p Q(1))^2 \geq (1-p) p Q(2)$, 
and this is exactly the assumption \eqref{cbern-lc}.

The assertion of the lemma now follows
since the sum of independent log-concave
random variables is log-concave; see, e.g., \cite{karlin3}.
\end{proof}

Next we examine conditions under which a compound
Poisson measure is log-concave, starting with a simple
necessary condition.

\begin{lemma}\label{lem:nec}
A necessary condition for {\em CPo}$(\lambda,Q)$
to be log-concave is that,
\begin{equation} \label{eq:nec}
\lambda \geq \frac{2 Q(2)}{Q(1)^2} .
\end{equation} 
\end{lemma}
\begin{proof}
For any distribution $P$, considering 
the difference,
$C_Q P(1)^2 - C_Q P(0) C_Q P (2)$,
shows that a necessary condition for $C_Q P$ 
to be log-concave is that,
\begin{equation} \label{eq:nec2}
(P(1)^2 - P(0) P(2))/P(0) P(1) \geq Q(2)/Q(1)^2. \end{equation}
Now take $P$ to be the Po$(\lambda)$
distribution.
\end{proof}

Similarly, for $P=\bp{p}$, a necessary condition
for the compound Bernoulli sum 
$C_Q\bp{p}$ to be log-concave is that,
$$ \sum_i \frac{ p_i}{1-p_i} + \left(\sum_i \frac{p_i^2}{(1-p_i)^2} \right)
\left(\sum_i \frac{p_i}{1-p_i} \right)^{-1} \geq \frac{2 Q(2)}{Q(1)^2},$$
which, since the left-hand-side
is greater than $\sum_i p_i/(1-p_i) \geq \sum_i p_i$,
will hold as long as
$\sum_i p_i \geq 2 Q(2)/Q(1)^2$.

Note that, unlike for the Poisson distribution, 
it is not the case that every compound Poisson 
distribution CPo$(\lambda,Q)$
is log-concave.


Next we show that for some particular choices of $Q$ 
and general compound distributions $C_Q P$,
the above necessary condition is sufficient
for log-concavity.

\begin{theorem} \label{thm:qgeom}
Let $Q$ be a geometric distribution
on $\Nat$. Then $C_Q P$ is log-concave
for any distribution $P$ which is log-concave
and satisfies the condition~{\em (\ref{eq:nec2})}.
\end{theorem}
\begin{proof} If $Q$ is geometric with mean $1/\alpha$, then,
$Q^{*y}(x) = \alpha^y (1-\alpha)^{x-y} \binom{x-1}{y-1}$,
which implies that,
$$ C_Q P(x) = \sum_{y=0}^x P(y) \alpha^y (1-\alpha)^{x-y} \binom{x-1}{y-1}.$$
Condition~(\ref{eq:nec2}) ensures 
that $C_Q P(1)^2 - C_Q P(0) C_Q P (2) \geq 0$, so,
taking $z = y-1$, we need only prove that the sequence,
$$ C(x) := C_Q P(x+1)/(1-\alpha)^x =  \sum_{z=0}^x P(z+1) \left(
\frac{\alpha}{1-\alpha} \right)^{z+1}  \binom{x}{z}$$
is log-concave.
However, this follows immediately from
\cite[Theorem~7.3]{karlin3}, which proves 
that if $\{a_i\}$ is a log-concave sequence, 
then so is $\{b_i\}$, defined by
$ b_i = \sum_{j=0}^i \binom{i}{j} a_j.$
\end{proof}

\begin{theorem} \label{thm:q2pt}
Let $Q$ be a distribution supported on the set $\{ 1, 2 \}$. 
Then the distribution $C_Q P$ is log-concave
for any ultra-log-concave distribution $P$
with support on $\{0,1,\ldots,N\}$
(where $N$ may be infinite),
which satisfies
\begin{equation}\label{12cond}
(x+1) P(x+1)/P(x) \geq
2 Q(2)/Q(1)^2
\end{equation} 
for all $x=0,1,\ldots,N$.

In particular, if $Q$ is supported on $\{ 1, 2 \}$, 
the compound Poisson distribution {\em CPo}$(\lambda,Q)$ 
is log-concave for all $\lambda \geq \frac{2 Q(2)}{Q(1)^2}$.
\end{theorem}

Note that the condition \eqref{12cond} is equivalent to requiring 
that $ NP(N)/P(N-1) \geq 2Q(2)/Q(1)^2$ if $N$ is finite, 
or that $\lim_{x \tends} (x+1) P(x+1)/P(x) 
\geq 2 Q(2)/Q(1)^2$ if $N$ is infinite.

The proof of Theorem~\ref{thm:q2pt} is based
in part on some of the ideas in
Johnson and Goldschmidt
\cite{johnson17}, and also
in Wang and Yeh \cite{wang3},
where transformations that 
preserve log-concavity are studied.
Since the proof is slightly involved and
the compound Poisson part of the theorem is superseded
by Theorem~\ref{thm:lcconj} below, we give it in the appendix.


Lemma~\ref{lem:nec} and Theorems~\ref{thm:qgeom} and \ref{thm:q2pt}, 
supplemented by some calculations of the quantities 
$\CPi_{\lambda}(x)^2 - \CPi_{\lambda}(x-1) \CPi_{\lambda}(x+1)$ 
for small $x$, suggest that 
compound Poisson measure {\em CPo}$(\lambda,Q)$ should
be log-concave, as long as $Q$ is log-concave 
and $\lambda Q(1)^2\geq 2Q(2)$.
Indeed, the following 
slightly more general result holds; see Section~7 for
some remarks on its history.
As per Definition~\ref{def:sizebias}, we use 
$Q^{\#}$ to denote the size-biased version
of $Q$. Observe that log-concavity of $Q^{\#}$ 
is a weaker requirement than log-concavity of $Q$. 

\begin{theorem}\label{thm:lcconj}
If $Q^{\#}$ is log-concave and $\lambda Q(1)^2\geq 2Q(2)$ with $Q(1)>0$, then 
the compound Poisson measure {\em CPo}$(\lambda,Q)$ is log-concave.
\end{theorem}
\begin{proof}
It is well-known that compound Poisson probability mass functions obey a 
recursion formula:
\begin{equation}\label{panjer}
k \CPi_{\lambda}(k) = \lambda \sum_{j=1}^{k} j Q(j) \CPi_{\lambda} (k-j) \;\;\;\;
\mbox{ for all $k\in \Nat$.}
\end{equation}
(This formula, which is easy to prove for instance 
using characteristic functions, has been 
repeatedly rediscovered; the earliest reference 
we could find was to a 1958 note of Katti and Gurland 
mentioned by N. de Pril \cite{deP85}, but later references 
are Katti \cite{Kat67}, 
Adelson \cite{Ade66} and Panjer \cite{Pan81}; 
in actuarial circles, the above is known
as the Panjer recursion formula.)
For notational convenience, we write $\mu_Q$ for the mean of $Q$,
$r_{j}=\lambda (j+1) Q(j+1)=\lambda \mu_Q Q^{\#}(j)$,
and $p_{j}=\CPi_{\lambda}(j)$ for $j\in\Zpl$. Then \eqref{panjer} reads,
$$
(k+1)p_{k+1}=  \sum_{j=0}^{k} r_j p_{k-j} 
$$
for all $k\in\Zpl$.

Theorem~\ref{thm:lcconj} is just a restatement using \eqref{panjer}
of \cite[Theorem 1]{Han88}. For completeness, we sketch the proof 
of Hansen \cite{Han88},
which proceeds by induction. 
Note that one only needs to prove the following
statement: If $Q^{\#}$ is strictly log-concave and $\lambda Q(1)^2> 2Q(2)$, 
then 
the compound Poisson measure CPo$(\lambda,Q)$ is strictly log-concave.
The general case follows by taking limits.

By assumption, $\lambda Q(1)^2> 2Q(2)$,
which can be rewritten as $r_0^2> r_1$,
and hence,
$$
p_1^2-p_0 p_2 = \frac{p_0^2}{2} (r_0^2 - r_1) > 0 .
$$
This initializes the induction procedure by showing that the subsequence
$(p_0,p_1,p_2)$ is strictly log-concave.
Hansen \cite{Han88} developed the following identity, 
which can be verified using
the recursion \eqref{panjer}:  Setting $p_{-1}=0$,
\begin{equation}\label{eq:hansen}\begin{split}
m(m+2) [p_{m+1}^2-p_{m} p_{m+2}]
&= p_{m+1} (r_0 p_{m}-p_{m+1}) \\
&\quad+ \sum_{l=0}^m \sum_{k=0}^l 
(p_{m-l}p_{m-k-1}-p_{m-k}p_{m-l-1}) (r_{k+1}r_{l}-r_{l+1}r_{k}) .
\end{split}\end{equation}
Observe that each term in the double sum is positive as a 
consequence of the induction hypothesis (namely, that the
subsequence $(p_0,p_1,\ldots,p_{m+1})$ is strictly log-concave),
and the strict log-concavity of $r$.
To see that the first term is also positive, note that the induction hypothesis
implies that $p_{k+1}/p_k$ is decreasing for $k\leq m+1$; 
hence,
$$
r_0=\frac{p_1}{p_0} >\frac{p_{m+1}}{p_m} .
$$
Thus it is shown that $p_{m+1}^2 > p_{m} p_{m+2}$, which proves the theorem.
\end{proof}

We note that Hansen's remarkable identity \eqref{eq:hansen}
is reminiscent of (although more complicated than) an identity
that can be used to prove the well-known fact that the convolution
of two log-concave sequences is log-concave. Indeed,
as shown for instance in Liggett \cite{Lig97}, if $c=a\star b$, then,
$$
c_k^2-c_{k-1}c_{k+1}
= \sum_{i<j} (a_i a_{j-1}-a_{i-1} a_{j}) (b_{k-i} b_{k-j+1}-b_{k-i+1} b_{k-j}).
$$
Observe that \eqref{panjer} can be interpreted as saying that the
size-biased version of $\CPi_{\lambda}$ is the convolution
of the sequence $r$ with the sequence $p$.

\section{Applications to Combinatorics}
\label{sec:applns}

There are numerous examples of ultra-log-concave sequences
in discrete mathematics, and also many examples of interesting
sequences where  ultra-log-concavity is conjectured. 
The above maximum entropy results for  ultra-log-concave 
probability distributions on $\Zpl$ yield bounds on the 
``spread'' of such ultra-log-concave sequences,
as measured by entropy.
Two particular examples are considered below.

\subsection{Counting independent sets in a claw-free graph}

Recall that for a graph $G=(V,E)$, an independent set is a subset of the
vertex set $V$ such that no two elements of the subset are connected by an
edge in $E$. The collection of independent sets of 
$G$ is denoted $\mathcal{I}(G)$.

Consider a graph $G$ on a randomly weighted
ground set, i.e., associate with each $i\in V$ the random weight
$X_i$ drawn from a probability distribution $Q$ on $\Nat$, and suppose
the weights $\{X_i:i\in V\}$ are independent. Then
for any independent set $I\in \mathcal{I}(G)$, its weight is given by
the sum of the weights of its elements,
$$
w(I)=\sum_{i\in I} X_i \eqD \sum_{i=1}^{|I|} X_i',
$$
where $\eqD$ denotes equality in distribution,
and $X_i'$ are i.i.d. random variables drawn from $Q$.
For the weight of a random independent set $\mathbb{I}$ (picked uniformly
at random from $\mathcal{I}(G)$), one similarly has,
$$
w(\mathbb{I})=\sum_{i\in \mathbb{I}} X_i \eqD \sum_{i=1}^{|\mathbb{I}|} X_i',
$$
and the latter, by definition, has distribution $C_Q P$, where
$P$ is the probability distribution on $\Zpl$ induced by $|\mathbb{I}|$.

The following result of Hamidoune \cite{Ham90}
(see also Chudnovsky and Seymour \cite{CS07} for a generalization
and different proof) connects this discussion with ultra-log-concavity.
Recall that a claw-free graph is a graph that does not
contain the complete bipartite graph $K_{1,3}$ as an induced subgraph.

\begin{theorem}[Hamidoune \cite{Ham90}]\label{thm:hamidoune}
For a claw-free finite graph $G$, the sequence $\{I_k\}$, 
where $I_k$ is the number of 
independent sets of size $k$ in $G$, is ultra-log-concave. 
\end{theorem}

Clearly, Theorem~\ref{thm:hamidoune} may be restated as follows: 
For a random independent set $\mathbb{I}$,
$$
P(k):=\text{Pr}\{|\mathbb{I}|=k\} \propto I_k,
$$
is an ultra-log-concave probability distribution. This yields the
following corollary.

\begin{corollary}\label{cor:graph}
Suppose the graph $G$ on the ground set $V$ is claw-free.
Let $\mathbb{I}$ be a random independent set,
and let the average cardinality of $\mathbb{I}$
be $\lambda$. Suppose the elements of the ground set are
given i.i.d. weights drawn from a probability distribution $Q$ on $\Nat$, 
where $Q$ is log-concave with $Q(1)>0$ and $\lambda Q(1)^2\geq 2Q(2)$.  
If $W$ is the random weight assigned to $\mathbb{I}$, then,
$$
H(W) \leq H(C_Q \Pi_{\lambda}).
$$
\end{corollary}

If $Q$ is the unit mass at 1, then $W=|\mathbb{I}|$, 
and Corollary~\ref{cor:graph}
gives a bound on the entropy of the cardinality of a
random independent set in a claw-free graph. That is,
\begin{equation*}
H(|\mathbb{I}|) \leq H(\Pi_{\lambda}) ,
\end{equation*}
where $\lambda=E|\mathbb{I}|$. Observe that this bound
is independent of $n$ and depends {\em only} on the average
size of a random independent set, which suggests that
it could be of utility in studying sequences associated with
graphs on large ground sets. 
And, although the entropy of a Poisson (or compound Poisson) 
measure cannot easily expressed in closed form, there are 
various simple bounds \cite[Theorem 8.6.5]{CT06:book} such as,
\begin{equation}\label{poi-ent-bd}
H(\Pi_{\lambda}) \leq \frac{1}{2} \log \bigg[2\pi e \bigg(\lambda+\frac{1}{12}\bigg)\bigg],
\end{equation}
as well as good approximations for large $\lambda$; see, e.g.,
\cite{Kne98,JS99,Fla99}.
One way to use this bound is via the
following crude relaxation: Bound the average
size $\lambda$ of a random independent set by the independence
number $\alpha(G)$ of $G$, which is defined as 
the size of a largest independent set of $G$. Then,
\begin{equation}
H(|\mathbb{I}|) \leq  \frac{1}{2} \log \bigg[2\pi e \bigg(\alpha(G)+\frac{1}{12}\bigg)\bigg],
\end{equation}
which can clearly be much tighter than the trivial bound,
$H(|\mathbb{I}|) \leq \log \alpha(G)$,
using the uniform distribution, when $\alpha(G)>16$.

\subsection{Mason's conjecture}

Recall that a matroid $M$ on a
finite ground set $[n]$ is a collection of subsets of $[n]$, 
called ``independent sets''\footnote{Note that although
graphs have associated cycle matroids, there is no connection 
between independent
sets of matroids and independent sets of graphs; indeed, the latter are often
called ``stable sets'' in the matroid literature to distinguish the two.},
satisfying the following:
(i) The empty set is independent. 
(ii) Every subset of an independent set is independent. 
(iii) If $A$ and $B$ are two independent sets and $A$ 
has more elements than $B$, 
then there exists an element in $A$ which is not in $B$ 
and when added to $B$ still 
gives an independent set.

Consider a matroid $M$ on a randomly weighted
ground set, i.e., associate with each $i\in[n]$ the random weight
$X_i$ drawn from a probability distribution $Q$ on $\Nat$, and suppose
the weights $\{X_i:i\in[n]\}$ are independent. As before,
for any independent set $I\in M$, its weight is given by
the sum of the weights of its elements,
and the weight of a random independent set $\mathbb{I}$ (picked uniformly
at random from $M$), is,
$$
w(\mathbb{I})=\sum_{i\in \mathbb{I}} X_i \eqD \sum_{i=1}^{|\mathbb{I}|} X_i',
$$
where the $X_i'$ are i.i.d.\ random variables drawn from $Q$.
Then $w(\mathbb{I})$
has distribution $C_Q P$, where
$P$ is the probability distribution on $\Zpl$ induced by $|\mathbb{I}|$.

\begin{conjecture}[Mason \cite{Mas72}]\label{conj:mason}
The sequence $\{I_k\}$, where $I_k$ is the number 
of independent sets of size $k$ in
a matroid on a finite ground set, is ultra-log-concave. 
\end{conjecture}

Strictly speaking, Mason's original conjecture asserts ultra-log-concavity of some finite order
(not defined in this paper) whereas this paper is only concerned with ultra-log-concavity of
order infinity; however the slightly weaker form of the conjecture
stated here is still difficult and open. The only special cases in which Conjecture~\ref{conj:mason} 
is known to be true is for matroids whose rank (i.e., cardinality of the largest independent 
set) is 6 or smaller (as proved by Zhao \cite{Zha85}), and for matroids on a ground set of 11 or
smaller (as proved by Kahn and Neiman \cite{KN11}). Very recently, Lenz \cite{Len11}
proved that the sequence $\{I_k\}$ is strictly log-concave, which is clearly a weak form of 
Conjecture~\ref{conj:mason}.

Conjecture~\ref{conj:mason} equivalently says
that, for a random independent set $\mathbb{I}$,
the distribution, $P(k)=
\text{Pr}\{|\mathbb{I}|~=~k\} \propto I_k,
$
is ultra-log-concave.  This yields the
following corollary.

\begin{corollary}\label{cor:matroid}
Suppose the matroid $M$ on the ground set $[n]$ satisfies Mason's conjecture.
Let $\mathbb{I}$ be a random independent set of $M$,
and let the average cardinality of $\mathbb{I}$
be $\lambda$. Suppose the elements of the ground set are
given i.i.d. weights drawn from a probability distribution $Q$ on $\Nat$, 
where $Q$ is log-concave and satisfies $Q(1)>0$ and $\lambda Q(1)^2\geq 2Q(2)$.  
If $W$ is the random weight assigned to $\mathbb{I}$, then,
$$
H(W) \leq H(C_Q \Pi_{\lambda}).
$$
\end{corollary}

Of course, if $Q$ is the unit mass at 1, 
Corollary~\ref{cor:matroid} gives (modulo Mason's conjecture) a bound 
on the entropy of the cardinality of a
random independent set in a matroid. That is,
\begin{equation*}
H(|\mathbb{I}|) \leq H(\Pi_{\lambda}) ,
\end{equation*}
where $\lambda=E|\mathbb{I}|$. 
As in the case of graphs, this bound is independent of $n$ and 
can be estimated in terms of the average size of a random independent set
(and hence, more loosely, in terms of the matroid rank)
using the Poisson entropy bound \eqref{poi-ent-bd}. 

\section{Extensions and Conclusions}
\label{sec:disc}

The main results in this paper describe the solution 
of a discrete entropy maximization problem, under both 
shape constraints involving log-concavity and constraints 
on the mean. Different entropy problems involving 
log-concavity of continuous densities have also been 
studied by Cover and Zhang \cite{cover2} 
and by Bobkov and Madiman \cite{BM11:it},
using different methods and motivated by 
different questions than those in this work.

The primary motivation for this work was the
development of an information-theoretic approach to discrete limit laws,
and specifically those corresponding to compound Poisson limits.
Above we have shown that, under appropriate conditions,
compound Poisson distributions have maximum entropy 
within a natural class. This is analogous 
to the maximum entropy property of the Gaussian and
Poisson measures, and their corresponding roles in
Gaussian and Poisson approximation, respectively.
Moreover, the techniques introduced here --
especially the introduction and analysis 
of a new score function that naturally connects
with the compound Poisson family -- turn out
to play a central role in the development of an
information-theoretic picture of compound Poisson 
limit theorems and approximation bounds
\cite{johnson22}.

After a preliminary version of this paper was made
publicly available \cite{jkm-arxiv},
Y.\ Yu \cite{Yu09:cp} provided different proofs 
of our Theorems~\ref{thm:mainpoi} and \ref{thm:mainber},
under less restrictive conditions, and using 
a completely different mathematical approach. 
Also, in the first version of \cite{jkm-arxiv},
motivated in part by the results of Lemma~\ref{lem:nec} 
and Theorems~\ref{thm:qgeom} and \ref{thm:q2pt}, 
we conjectured that the compound Poisson measure 
CPo$(\lambda,Q)$ is log-concave,
if $Q$ is log-concave and $\lambda Q(1)^2\geq 2Q(2)$.  
Y. Yu \cite{Yu09:cp} subsequently established 
the truth of the conjecture by pointing out 
that it could be proved by an application
of the results of Hansen in \cite{Han88}.
Theorem~\ref{thm:lcconj} in Section~5 is 
a slightly more general version of that
earlier conjecture. Note that in order 
to prove the conjecture it is not necessary
to reduce the problem to the strictly log-concave
case (as done in the proof 
of Theorem~\ref{thm:lcconj}),
because the log-concavity of $Q$ implies 
a bit more than log-concavity for $Q^{\#}$.
Indeed, the following variant of 
Theorem~\ref{thm:lcconj} is easily proved: 
If $Q$ is log-concave with $Q(1)>0$ and $\lambda Q(1)^2> 2Q(2)$, 
then the compound Poisson measure 
CPo$(\lambda,Q)$ is strictly log-concave.

In closing we mention another possible
direction in which the present results 
may be extended.
Suppose that the compounding distribution $Q$
in the setup described in Section~2 is supported
on $\mathbb R$ and has a density with respect to
Lebesgue measure. The definition of compound distributions 
$C_Q P$ (including compound Poissons)
continues to make sense for probability distributions 
$P$ on the nonnegative integers, but these now clearly 
are of mixed type, with a continuous component
and an atom at 0. Furthermore, 
limit laws for sums converging to such mixed-type compound Poisson
distributions hold exactly as in the discrete case. 
It is natural and interesting to ask for such
`continuous' analogs of the present maximum entropy 
results, particularly as neither their form nor method 
of proof are obvious in this case.

\section*{Acknowledgement}
We wish to thank Zhiyi Chi for sharing his unpublished compound binomial 
counter-example mentioned in equation~(\ref{eq:chi}),
and David G. Wagner and Prasad Tetali for useful comments. 
Some of the ideas leading to the combinatorial
connections described in Section~6 were sparked by the participation
of the third-named author in the Workshop on Combinatorial and Probabilistic
Inequalities at the Isaac Newton Institute for Mathematical Sciences in 
Cambridge, UK, in June 2008,
and in the Jubilee Conference for Discrete Mathematics
at the Banasthali Vidyapith in Rajasthan, India, in January 2009; 
he expresses his gratitude to the 
organizers of both these events for their hospitality.

\appendix
\section{Appendix}

\begin{proof}
Writing $R(y) = y! P(y)$, we know that
$ C_Q P(x) = \sum_{y = 0}^x R(y) \left( Q^{*y}(x)/y! \right).$
Hence, the log-concavity of $C_QP(x)$ is equivalent 
to showing that, 
\begin{equation}
\sum_r \frac{ Q^{*r}(2x)}{r!}
\sum_{y+z = r} R(y) R(z) \binom{r}{y} 
\left( \frac{ Q^{*y}(x) Q^{*z}(x)}{Q^{*r}(2x)}
- \frac{ Q^{*y}(x+1) Q^{*z}(x-1)}{Q^{*r}(2x)} \right)\geq 0,
\label{eq:toabel}
\end{equation}
for all $x\geq 2$, 
since the case of $x = 1$ was 
dealt with previously
by equation~(\ref{eq:nec2}).
In particular, for~(i), taking
$P=\mbox{Po}(\lambda)$, 
it suffices to show 
that for all $r$ and $x$, the 
function,
$$ g_{r,x}(k) := 
\sum_{y+z = r} \binom{r}{y} \frac{ Q^{*y}(k) Q^{*z}(2x-k)}{Q^{*r}(2x)} $$
is unimodal as a function of $k$ 
(since $g_{r,x}(k)$ is symmetric about $x$).

In the general case~(ii), writing $Q(2)=p=1-Q(1)$, we have,
$Q^{*y}(x) = \binom{y}{x-y} p^{x-y} (1-p)^{2y-x}$,
so that,
\begin{equation} \label{eq:exact}
\binom{r}{y} \frac{ Q^{*y}(k) Q^{*z}(2x-k)}{Q^{*r}(2x)} = \binom{2x-r}{k-y} \binom{2r - 2x}{2y-k},
\end{equation}
for any $p$.
Now, following
\cite[Lemma~2.4]{johnson17} and 
\cite[Lemma~2.1]{wang3}, we use summation by parts 
to show that the inner 
sum in (\ref{eq:toabel})
is positive for each $r$ (except for $r=x$ when $x$ is odd), 
by case-splitting according to the parity of $r$.

(a) For $r = 2t$, we rewrite the inner sum 
of equation~(\ref{eq:toabel}) as,
\begin{eqnarray*} 
\lefteqn{ \sum_{s = 0}^t ( R(t+s) R(t-s) - R(t+s+1) R(t-s-1) ) \times  } \\
& & \left( \sum_{y = t-s}^{t+s} \left( \binom{2x-r}{x-y} \binom{2r-2x}{2y-x} -
\binom{2x-r}{x+1-y} \binom{2r-2x}{2y-x-1} \right) \right), \end{eqnarray*}
where the first term in the
above product is positive by the ultra-log-concavity of $P$ 
(and hence log-concavity of $R$), and the second term is positive 
by Lemma~\ref{lem:tech2} below.

(b) Similarly, for $x\neq r = 2t+1$, we rewrite the inner 
sum of equation~(\ref{eq:toabel}) as,
\begin{eqnarray*}  \lefteqn{ \sum_{s = 0}^t ( R(t+s+1) R(t-s) - R(t+s+2) R(t-s-1)) \times  } \\
& &  \left( \sum_{y = t-s}^{t+1+s} \left( \binom{2x-r}{x-y} \binom{2r-2x}{2y-x} -
\binom{2x-r}{x+1-y} \binom{2r-2x}{2y-x-1} 
\right) \right),\end{eqnarray*}
where the first term in the product
is positive by the ultra-log-concavity of $P$ 
(and hence log-concavity of $R$) and the second term 
is positive by 
Lemma~\ref{lem:tech2} below.

(c) Finally, in the case of $x=r=2t+1$, 
substituting $k = x$ and $k = x+1$
in (\ref{eq:exact}), combining the resulting
expression with
(\ref{eq:toabel}), and noting
that $\binom{2r-2x}{u}$ is 1 if and only if $u=0$
(and is zero, otherwise), we see that the 
inner sum becomes,
$- R(t+1) R(t) \binom{2t+1}{t}$,
and the summands in 
(\ref{eq:toabel}) reduce to,
$$ - \frac{ {p^x} R(t) R(t+1)}{(t+1)! t!}.$$
However, the next term in the outer sum of equation~(\ref{eq:toabel}), 
$r = x+1$, gives
\begin{eqnarray*}
\lefteqn{ \frac{ p^{x-1} (1-p)^2}{2(2t)!} 
\left[ R(t+1)^2 \left( 2 \binom{2t}{t} - \binom{2t}{t+1} \right)
- R(t) R(t+2) \binom{2t}{t} \right] } \\
& \geq & \frac{ p^{x-1} (1-p)^2}{2 (2t)!} 
R(t+1)^2 \left(  \binom{2t}{t} - \binom{2t}{t+1} \right)
= \frac{p^{x-1} (1-p)^2}{ 2(t+1)! t!} R(t+1)^2.
\end{eqnarray*}
Hence, the sum of the first two terms is positive (and hence the whole sum
is positive) if $R(t+1) (1-p)^2/(2p) \geq R(t)$.

If $P$ is Poisson($\lambda$), this simply reduces to equation~(\ref{eq:nec}), 
otherwise we use the fact
that $R(x+1)/R(x)$ is decreasing.
\end{proof}

\begin{lemma} \label{lem:tech2}

{\rm (a)} If $r = 2t$, for any $0 \leq s \leq t$, the sum,
$$ \sum_{y = t-s}^{t+s} \left( \binom{2x-r}{x-y} \binom{2r-2x}{2y-x} -
\binom{2x-r}{x+1-y} \binom{2r-2x}{2y-x-1} \right) 
\geq 0.$$

{\em (b)} If $x\neq r = 2t+1$, for any $0 \leq s \leq t$, the sum,
$$ \sum_{y = t-s}^{t+1+s} \left( \binom{2x-r}{x-y} \binom{2r-2x}{2y-x} -
\binom{2x-r}{x+1-y} \binom{2r-2x}{2y-x-1} 
\right) \geq 0.$$
\end{lemma}
\begin{proof} The proof is in two stages; first we show that 
the sum is positive for $s = t$, then we show that there 
exists some $S$ such that, as $s$ increases,
the increments are positive for $s \leq S$ 
and negative for $s > S$. The result
then follows, as in \cite{johnson17} or \cite{wang3}.

For both (a) and (b), note that for $s =t$, 
equation~(\ref{eq:exact}) implies that 
the sum is the difference between the coefficients of 
$T^x$ and $T^{x+1}$ in $f_{r,x}(T) = (1+T^2)^{2x-r} (1+T)^{2r-2x}$.
Since $f_{r,x}(T)$ has degree $2x$ and has coefficients 
which are symmetric about $T^x$,
it is enough to show that the coefficients 
form a unimodal sequence. Now, $(1+T^2)^{2x-r} (1+T)$ has
coefficients which do form a unimodal sequence. 
Statement $S_1$ of Keilson and Gerber \cite{keilson} states
that any binomial distribution is strongly unimodal, 
which means that it preserves unimodality on
convolution. This means that 
$(1+T^2)^{2x-r} (1+T)^{2r-2x}$ is unimodal if $r-x \geq 1$,
and we need only check the case $r=x$, when $f_{r,x}(T) = (1+T^2)^r$.
Note that if $r = 2t$ is even, the difference between 
the coefficients of $T^{x}$ and
$T^{x+1}$ is $\binom{2t}{t}$, which is positive.

In part (a), the increments are equal to
$\binom{2x-2t}{x-t+s} \binom{4t-2x}{2t-2s-x}$ multiplied
by the expression,
\begin{eqnarray*}
2 - \frac{ (x-t-s)(2t-2s-x)}{(x+1-t+s)(2t+2s-x+1)}
- \frac{ (x-t+s)(2t+2s-x)}{(x+1-t-s)(2t-2s-x+1)},
\end{eqnarray*}
which is positive for $s$ small and negative for $s$ large, since 
placing the term in brackets over a common denominator, 
the numerator is of the form $(a-bs^2)$.

Similarly, in part (b), the increments equal
$\binom{2x-2t-1}{x-t+s} \binom{4t+2-2x}{2t-2s-x} $ times
the expression,
\begin{eqnarray*}
2 - \frac{ (x-t-s-1)(2t-2s-x)}{(x+1-t+s)(2t+2s-x+3)}
- \frac{ (x-t+s)(2t+2+2s-x)}{(x-t-s)(2t+1-2s-x)},
\end{eqnarray*}
which is again positive for $s$ small and negative for $s$ large.
\end{proof}


\bibliographystyle{elsart-num}

\end{document}